\documentclass[12pt]{amsart}

\usepackage[text={16cm,22cm}]{geometry}

\usepackage{amsmath}
\usepackage{amssymb}
\usepackage{latexsym}
\usepackage{graphicx,color}
\usepackage{booktabs}
\usepackage{epsfig}
\usepackage{url}
\usepackage[boxed,vlined]{algorithm2e}

\begin{document}
\theoremstyle{plain}
\newtheorem{theorem}{Theorem}
\newtheorem{prop}[theorem]{Proposition}
\newtheorem{cor}[theorem]{Corollary}
\newtheorem{lemma}[theorem]{Lemma}
\newtheorem{question}[theorem]{Question}
\newtheorem{conj}[theorem]{Conjecture}
\newtheorem{assumption}[theorem]{Assumption}

\theoremstyle{definition}
\newtheorem{definition}[theorem]{Definition}
\newtheorem{notation}[theorem]{Notation}
\newtheorem{condition}[theorem]{Condition}
\newtheorem{example}[theorem]{Example}
\newtheorem{introduction}[theorem]{Introduction}
\newtheorem{remark}[theorem]{Remark}

\providecommand{\abs}[1]{\left\lvert#1\right\rvert}
\providecommand{\norm}[1]{\left\lVert#1\right\rVert}
\providecommand{\floor}[1]{\left\lfloor#1\right\rfloor}
\providecommand{\Z}{\mathbb{Z}} \providecommand{\R}{\mathbb{R}}
\providecommand{\T}{\mathbb{T}} \providecommand{\PP}{\mathbb{P}}
\providecommand{\N}{\mathbb{N}} \providecommand{\C}{\mathbb{C}}
\providecommand{\Q}{{\mathbb{Q}}} \providecommand{\x}{\mathbf{x}}
\providecommand{\y}{\mathbf{y}} \providecommand{\z}{\mathbf{z}}
\providecommand{\boxend}{\hspace{\stretch{1}}$\Box$\\ \ \\}

\providecommand{\SL}{\operatorname{SL}}
\providecommand{\GL}{\operatorname{GL}}
\providecommand{\Bsph}{B_{\mathrm{sph}}}
\providecommand{\Nsph}{N_{\mathrm{sph}}}
\providecommand{\Tsph}{T_{\mathrm{sph}}}
\providecommand{\Wsph}{W_{\mathrm{sph}}}
\providecommand\SetOf[2]{\left\{\, #1\vphantom{#2}\,:\,\vphantom{#1}#2 \,\right\}}
\providecommand\smallSetOf[2]{\{#1 : #2\}}

\providecommand\conv{\operatorname{conv}}
\providecommand\val{\operatorname{val}}
\providecommand\type{\operatorname{type}}
\providecommand\vtype{\operatorname{vtype}}
\providecommand\tconv{\operatorname{tconv}}
\providecommand\maxconv{\operatorname{maxconv}}
\providecommand\minconv{\operatorname{minconv}}
\providecommand\image{\operatorname{image}}
\providecommand\kernel{\operatorname{kernel}}
\providecommand\diag{\operatorname{diag}}
\providecommand\BT{\mathcal{B}}
\providecommand\TP{\T\PP}
\providecommand\cA{\mathcal{A}}
\providecommand\cC{\mathcal{C}}
\providecommand\cF{M}
\providecommand\Hom{\operatorname{Hom}}
\providecommand\vertices{\operatorname{vert}}

\providecommand{\germ}{\mathfrak}

\renewcommand\theenumi{\alph{enumi}}

\SetArgSty{rm}

\providecommand\comment[1]{\marginpar{\raggedright\tiny #1}}

\title{Affine Buildings and Tropical Convexity}

\author[Joswig, Sturmfels, and Yu]{Michael Joswig \and Bernd Sturmfels \and Josephine Yu}

\address{Michael Joswig, FB Mathematik, AG~7, TU Darmstadt, 64289 Darmstadt, Germany}
\email{joswig@mathematik.tu-darmstadt.de}
\address{Bernd Sturmfels,  Department of Mathematics, UC Berkeley, Berkeley CA 94720, USA}
\email{bernd@math.berkeley.edu}
\address{Josephine Yu, Department of Mathematics, M.I.T., Cambridge MA 02139, USA}
\email{jyu@math.mit.edu}

\begin{abstract}
  The notion of convexity in tropical geometry is closely related to notions of convexity in the theory of affine
  buildings.  We explore this relationship from a combinatorial and computational perspective.  Our results include a
  convex hull algorithm for the Bruhat--Tits building of $\SL_d(K)$ and techniques for computing with apartments and
  membranes.  While the original inspiration was the work of Dress and Terhalle in phylogenetics, and of Faltings,
  Kapranov, Keel and Tevelev in algebraic geometry, our tropical algorithms will also be applicable to problems in other
  fields of mathematics.
 \end{abstract}

\maketitle

\section{Introduction}

\noindent
Buildings were initially introduced by Tits \cite{Ti} to provide a common geometric framework for all simple Lie groups,
including those of exceptional type. The later work of Bruhat and Tits~\cite{BT} showed that buildings are fundamental
in a much wider context, for instance, for applications in arithmetic algebraic geometry.  Among the affine buildings,
the key example is the \emph{Bruhat--Tits building} $\BT_d$ of the special linear group $\SL_d(K)$ over a field $K$ with
a discrete non-archimedean valuation.  An active line of research explores compactifications of the building~$\BT_d$;
for example, see Kapranov~\cite{Ka} and Werner~\cite{We1,We2}.

Our motivation to study affine buildings stems from the connection to biology which was proposed in Andreas Dress' 1998
ICM lecture \emph{The tree of life and other affine buildings} \cite{DT2}.  Dress and Terhalle \cite{DT1} introduced
\emph{valuated matroids} as a combinatorial approximation of the building $\BT_d$, thereby generalizing the familiar
one-dimensional picture of an infinite tree for $d=2$.  In Section \ref{sec:membranes} we shall see that valuated
matroids are equivalent to the \emph{matroid decompositions} of hypersimplices of Kapranov~\cite[Definition 1.2.17]{Ka},
to the \emph{tropical linear spaces} of Speyer \cite{Spe1}, and to the \emph{membranes} of Keel and Tevelev \cite{KT}.
The latter equivalence, shown in \cite[Theorem~4.15]{KT}, will be revisited in Theorem~\ref{thm:membrane} below.

We start out in Section \ref{sec:BT} with a brief introduction to the Bruhat--Tits building $\BT_d$ 
and to the notion of convexity in $\BT_d$ which appears in work of Faltings \cite{Fal}.
For sake of concreteness we take $K$ to be the field $\C(\!(z)\!)$ of formal Laurent series with complex
coefficients.  Our discussion revolves around the algorithmic problem of computing the convex hull of a finite set of points in the building $\BT_d$.  Here each point is a lattice which is represented by an invertible $d \times d$-matrix with
entries in $K = \C(\!(x)\!)$.  Our solution to this problem involves identifying their convex hull in $\BT_d$ with a certain \emph{tropical polytope}.

Tropical convexity was introduced by Develin and Sturmfels~\cite{DS}.  Tropical polytopes are certain contractible
polytopal complexes which are dual to the regular polyhedral subdivisions of the product of two simplices.  A review of
tropical convexity will be given in Section \ref{sec:membranes}, along with some new results, extending a formula of
Ardila \cite{Ard}, which characterize the nearest point projection onto a tropical polytope. In Section
\ref{sec:membranes}, we introduce tropical linear spaces, we represent them as tropical polytopes, and we identify them
with membranes in $\BT_d$. This allows us in Section \ref{sec:algorithms} to reduce convexity in $\BT_d$ to tropical
convexity. In addition to our convex hull algorithm, we also study the related problems of intersecting apartments or,
more generally, membranes. We prove the following result:

\begin{theorem}
  The min- and max-convex hulls of a finite set of lattices in $\BT_d$ coincides with the standard triangulation of a
  tropical polytope in a suitable membrane.
\end{theorem}

This is stated more precisely in Proposition \ref{prop:Algo1}.
New contributions made by this paper include the triangulation of
tropical polytopes in Theorem \ref{thm:triang}, the formulas for projecting onto tropical linear spaces in Theorem
\ref{thm:blue-red}, a combinatorial proof for the Keel-Tevelev bijection in Theorem \ref{thm:membrane}, and, most
important of all, the algorithms in Sections \ref{sec:algorithms} and \ref{sec:perspectives}.

\bigskip

\noindent \textbf{Acknowledgments:}
Michael Joswig was partially supported by Deu\-tsche For\-schungs\-ge\-mein\-schaft (FOR 565 \emph{Polyhedral Surfaces}).
Bernd Sturmfels was partially supported by the National Science Foundation (DMS-0456960), and Josephine Yu was supported
by a UC Berkeley Graduate Opportunity Fellowship and by the IMA in Minneapolis.

\smallskip

\section{The Bruhat--Tits building of $\SL_d(K)$}\label{sec:BT}

\noindent
We review basic definitions concerning Bruhat--Tits buildings, following the presentations in \cite{KT, Mus}.  The most
relevant section in the monograph by Abramenko and Brown is \cite[\S 6.9]{Bro}.  Let $R = \C[\![z]\!]$ be the ring of
formal power series with complex coefficients. Its field of fractions is the field $K = \C(\!(z)\!)$ of formal Laurent
series with complex coefficients.  Taking the exponent of the lowest term of a power series defines a valuation $\val :
K^* \rightarrow \Z$.  Note that $R$ is the subring of $K$ consisting of all field elements $c$ with $\val(c) \geq 0$.
What follows is completely general and works for other fields with a non-archimedean discrete
valuation, notably the $p$-adic numbers, but to keep 
matters most concrete we fix $K = \C(\!(z)\!)$.  We extend the
valuation to $K$ by setting $\val(0) = \infty $.  If $M$ is a matrix over $K$ then $\val(M)$ 
denotes the matrix over $\Z \cup
\{\infty\}$ whose entries are the values of the entries of $M$.

The vector space $K^d$ is a module over the ring $R$.  A \emph{lattice} in $K^d$ is an $R$-submodule generated by $d$
linearly independent vectors in $K^d$.  Each lattice $\Lambda$ is represented as the image of a matrix $M$ with $d$ rows
and $\geq d$ columns, with entries in $K$, having rank $d$.  Two lattices $\Lambda_1, \Lambda_2 \subset K^d$ are
\emph{equivalent} if $c \Lambda_1 = \Lambda_2$ for some $c \in K^*$.
Two equivalence classes of lattices are called \emph{adjacent} if there are representatives $\Lambda_1$ and $\Lambda_2$
such that $z \Lambda_2 \subset \Lambda_1 \subset \Lambda_2$.

The \emph{Bruhat--Tits building} of $\SL_d(K)$ is the flag simplicial complex $\BT_d$ whose vertices are the
equivalence classes of lattices in $K^d$ and whose edges are the adjacent pairs of lattices.  Being a \emph{flag
  simplicial complex} means that a finite set of vertices forms a simplex if and only if any two elements in that set
form an edge.  The link of any lattice $\Lambda$ in $\BT_d$ is isomorphic to the simplicial complex of all
chains of subspaces in $\C^d = \Lambda/z\Lambda$.  Thus the simplicial complex $\BT_d$ is pure of dimension
$d-1$, but it is not locally finite, since the residue field is~$\C$.  Our objective is to identify finite
subcomplexes with a nice combinatorial structure
 which is suitable for reducing computations in $\BT_d$ to tropical geometry.
  
If $\Lambda_1$ and $\Lambda_2$ are lattices then their $R$-module sum $\Lambda_1 + \Lambda_2$ and their intersection
$\Lambda_1 \cap \Lambda_2$ are also lattices.  These two operations give rise to two different notions of convexity on the Bruhat--Tits building $\BT_d$.  We say that a set
 $\mathcal{M}$ of lattices in $\BT_d$ is \emph{max-convex} if
the set of all representatives for lattices in $\mathcal{M}$ is closed under finite $R$-module sums.  
We call $\mathcal{M}$
\emph{min-convex} if that set is closed under finite intersections. If $\mathcal{L}$ is
any subset of $\BT_d$ then its \emph{max-convex hull} 
$\maxconv(\mathcal{L})$ is the set of all lattices $\Lambda$ in
$K^d$ such that $\Lambda$ is the $R$-module sum of finitely 
many lattices in $\mathcal{L}$. Similarly, the {\rm
  min-convex hull} $\minconv(\mathcal{L})$ is the set of 
  all lattices $\Lambda$ in $K^d$ such that $\Lambda$ is the intersection
of finitely many lattices in $\mathcal{L}$.  These notions of convexity give rise to the following 
problem in computational algebra:

\smallskip

\noindent \textbf{Computational Problem A}.
Let $M_1,\ldots,M_s$ be invertible $d \times d$-matrices with entries in  $K = \C(\!(z)\!)$, 
representing
lattices $\Lambda_i = \image_R(M_i)$ in $K^d$.  Compute both the min-convex hull
and the max-convex hull of the lattices $\Lambda_1,\ldots,\Lambda_s$
in the Bruhat--Tits building $\BT_d$.

\smallskip 

The duality functor $\Hom_R(\,\,\cdot \,\,, R)$
reduces a min-convex hull computation to a
max-convex hull computation and vice versa.
Given any lattice $\Lambda$, we write $\,\Lambda^* = \Hom_R(\Lambda, R)\,$
for the dual lattice.  Any $R$-module homomorphism $\Lambda \rightarrow R$ extends
uniquely to a $K$-vector
space homomorphism $K^d \rightarrow K$.  Hence the free $R$-module $\Lambda^*$ can be considered as a
lattice in the dual vector space $(K^d)^* = \Hom_R(K^d, K) $, consisting of those elements that send $\Lambda$
into $R$.  For any unit $c \in K^*$, we have $\,(c \Lambda)^* = \frac{1}{c}(\Lambda^*)$.  Since duality is
inclusion-reversing, i.e.~$\Lambda_1 \subset \Lambda_2$ implies $ \Lambda_2^* \subset \Lambda_1^*$, it respects
equivalence of lattices and adjacency of vertices in the building $\BT_d$.  Moreover, duality switches sums
and intersections:

\begin{lemma} \label{lem:dual}
For any two lattices $\Lambda_1, \Lambda_2$ in $K^d$, we have 
$\,(\Lambda_1 + \Lambda_2)^* \,=  \, \Lambda_1^* \cap \Lambda_2^* \,$ in $\,(K^d)^*$. 
\end{lemma}

\begin{proof}
  The inclusion ``$\subseteq$" is given by restricting any ring homomorphism $\,\phi:\Lambda_1 + \Lambda_2 \rightarrow
  R\,$ to $\Lambda_1 $ and to $\Lambda_2$, respectively.  The reverse inclusion ``$\supseteq$" is given by identifying
  $\,\phi \in \Lambda_1^* \cap \Lambda_2^*\,$ with the map $f_1 + f_2 \mapsto \phi(f_1) + \phi(f_2)$ where $f_i \in
  \Lambda_i$.
\end{proof}

It is known that both the max-convex hull and the min-convex hull of $\Lambda_1,\ldots,\Lambda_s$ are finite simplicial
complexes of dimension $\leq d-1$.  This finiteness result is attributed by Keel and Tevelev \cite[Lemma 4.11]{KT} to
Faltings' paper on matrix singularities \cite{Fal}.
 
Our usage of the prefixes ``min'' and ``max'' for convexity in $\BT_d$ is consistent with the alternative
representation of the Bruhat--Tits building in terms of additive norms on $K^d$.  
An \emph{additive norm} is a
map $N : K^d \rightarrow \R \cup \{\infty\}$ which satisfies the following three axioms:
\begin{enumerate}
\item $N(c \cdot f) = \val(c) + N(f)$ for any $c \in K$ and $f \in K^d$,
\item $N(f + g) \geq \min(N(f),N(g))$ for any $f,g \in K^d$,
\item $N(f) = \infty$ if and only if $f  = 0$.
\end{enumerate}
We say that $N$ is an \emph{integral} additive norm if $N$ takes values in $\Z \cup \{\infty\}$.

There is a natural bijection between lattices in $K^d$ and integral additive norms on $K^d$.  
 Namely, if $N$ is an integral additive norm then its lattice is
$\,\Lambda_N = \{ f \in K^d : N(f) \geq 0 \}$.  Conversely, if $\Lambda$ is any lattice
in $K^d$ then its additive norm
$N_\Lambda $ is given by
\begin{equation}
\label{norm}
N_\Lambda(f) \  := \  \max\SetOf{ u \in \Z }{ z^{-u} f \in \Lambda } \  =  \  \min(\val(M^{-1}f))\, ,
\end{equation}
where $M$ is a $d\times d$-matrix whose columns form a basis for $\Lambda$.  This bijection induces a homeomorphism
between the space of all additive norms (with the topology of pointwise convergence) and the space underlying the
Bruhat--Tits building $\BT_d$.  In other words, non-integral additive norms can be identified with points in the
simplices of $\BT_d$.

If $\Lambda_1$ and $\Lambda_2$ are lattices 
then the additive norm corresponding to the intersection $\Lambda_1 \cap \Lambda_2$
is the pointwise minimum of the two norms:
\[
N_{\Lambda_1 \cap \Lambda_2} \quad = \quad \min( N_{\Lambda_1}, N_{\Lambda_2}) \, .
\]
The pointwise maximum of two additive norms is generally not an additive norm. We write
$\overline{\max}(N_{\Lambda_1},N_{\Lambda_2})$ for the smallest norm which is pointwise greater than or equal to
$\max(N_{\Lambda_1},N_{\Lambda_2})$. Then we have
\[
N_{\Lambda_1 + \Lambda_2} \quad = \quad \overline{\max}(N_{\Lambda_1},N_{\Lambda_2}) \, .
\]
Our two notions of convexity on $\BT_d$ correspond to the $\min$ and
the $\overline\max$ of additive norms.
We now present a one-dimensional example which illustrates
Computational Problem~A.

\begin{example}[The convex hull of four $2 \times 2$-matrices]
\label{4leaftree} 
We consider eight vectors in $K^2$:
\[
a = \binom{z^{-3}}{z^{-3}},\,\,
b = \binom{z^{-4}}{z^5},\,\,
c = \binom{z^3}{z},\,\,
d = \binom{z^{-1}}{z^{-1}},
\]
\[
e = \binom{z^2}{z^3},\,\,
f = \binom{z^4}{z^{-4}},\,\,
g = \binom{z}{1},\,\,
h = \binom{z^4}{z}.
\]
We compute the min-convex hull in $\BT_2$ of the four lattices
\[
\Lambda_1 = R \{a,b\},\,\,\, \Lambda_2 = R \{c,d\},\,\,\, \Lambda_3 = R \{e,f\},\,\,\, \Lambda_4 = R \{g,h\}.
\]
The Bruhat--Tits building $\BT_2$ is an infinite tree \cite[\S6.9.2]{Bro}, and
 $\minconv(\Lambda_1,\Lambda_2,\Lambda_3,\Lambda_4)$ is a subtree with four
leaves and seven interior nodes, as shown in Figure~\ref{Fig1}.
The $11$ nodes in this tree represent the equivalence classes of lattices in the min-convex hull of $\Lambda_1,\Lambda_2,\Lambda_3,\Lambda_4$.  Our 
Algorithm~\ref{algo:A:minconvex} outputs a
representative lattice for each of the $11$ classes:
\begin{eqnarray*}
(1, 0, 7, 3, 6, 6, 5, 8) &
\{af, bf, cf, df, ef, fg, fh\}  \\
(1, 0, 7, 3, 6, 5, 5, 8) &
\{af, bf, cf, df, ef, fg, fh\} \\
(1, 0, 7, 3, 6, 4, 5, 8) &
\{af, bf, cf, df, ef, fg, fh\} \\
(1, 0, 7, 3, 6, 3, 5, 8) &
\{af, ah, bf, bh, cf, ch, df, dh, ef, eh, fg, fh, gh \} \\
(1, 0, 7, 3, 6, 2, 5, 7) &
\{ac, af, ah, bc, bf, bh, cd, ce, cf, cg, ch, df, \ldots, gh\} \\
(1, 0, 6, 3, 6, 1, 5, 6) &
\{ac, af, ag, ah, bc, bf, bg, bh, cd, ce, cg, df, \ldots, gh\} \\
(1, 0, 6, 3, 6, 1, 6, 6) &
\{ag, bg, cg, dg, eg, fg, gh\} \\
(1, 0, 5, 3, 6, 0, 4, 5) &
\{ab, ac, ae, af, ag, ah, bc, bd, bf, bg, bh, cd, \ldots, eh\} \\
(1, 1, 5, 3, 7, 0, 4, 5) &
\{ab, ae, bc, bd, be, bf, bg, bh, ce, de, ef, eg, eh\} \\
(2, 0, 5, 4, 6, 0, 4, 5) & 
\{ab, ac, ae, af, ag, ah, bd, cd, de, df, dg, dh\} \\
(3, 0, 5, 5, 6, 0, 4, 5) &
\{ab, ac, ae, af, ag, ah, bd, cd, de, df, dg, dh\} 
\end{eqnarray*}
Each of the $11$ lattices is represented by a vector
$u$ in $\mathbb{N}^8$ followed by
a set of pairs from $\{a,b,c,d,e,f,g,h\}$.
This data represents the following lattice
in  $\,\minconv(\Lambda_1,\Lambda_2,\Lambda_3,\Lambda_4)$:
\[
\Lambda \,\, = \,\, 
R \{ z^{-u_1}   a , \,
z^{-u_2}   b , \,
z^{-u_3}   c , \,
z^{-u_4}   d , \,
z^{-u_5}   e , \,
z^{-u_6}   f , \,
z^{-u_7}   g , \,
z^{-u_8}   h 
\}.
\]
Certain pairs among the eight generators form bases of $\Lambda \cong R^2$. The list of pairs indicates these bases.
For example, the fourth-to-last row $\,(1, 0, 5, 3, 6, 0, 4, 5) \,\, \ldots  \,$ represents 
\[
    R \{z^{-1} a,  b\}  = 
    R\{ z^{-1} a, z^{-5} c  \}  = 
    R\{ z^{-1} a, z^{-6} e  \} =    \,  \cdots \,  = \,
    R\{ z^{-6} e, z^{-5} h \} .
\]
The class of this lattice corresponds to the trivalent node on the right in Figure~1.

\begin{figure}\centering\label{Fig1}
  \includegraphics[scale = 0.85]{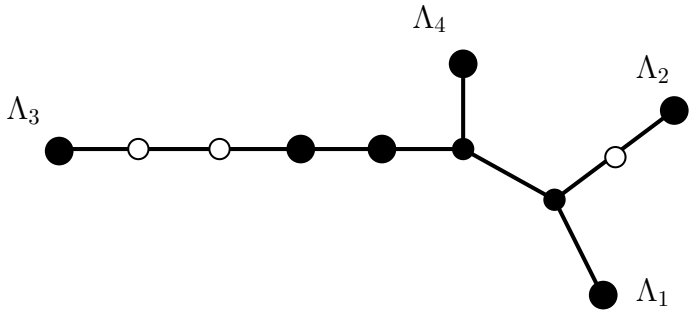}
  \caption{The convex hull of four points in the building $\BT_2$.}
\end{figure}

The bases can be determined from the labels of the arrows in Figure~\ref{Fig2}. A node uses a basis if and only if the
node lies on the two-sided infinite path (or \emph{apartment}) spanned by those arrows. There are eight distinct sets of
pairs appearing in the above list, indicating that the tree in Figure~\ref{Fig1} is divided into eight cells.  This
subdivision is the key ingredient in our algorithm.  \qed
\end{example}

\smallskip

Returning to our general discussion, we fix an arbitrary finite subset
 $M = \{f_1,\ldots,f_n\}$  of $K^d$ 
 which spans $K^d$ as a $K$-vector space, and we consider
the set of all equivalence classes of lattices of the form
$\,
\Lambda \, = \, R \bigl\{ z^{-u_1} f_1,\,  z^{-u_2} f_2,\,  z^{-u_3} f_3,\, \ldots ,\,z^{-u_n} f_n \bigr\},
$
where $ u_1,u_2,\ldots,u_n$ are any integers. This set of lattice classes is called the \emph{membrane} spanned by
$M$ in the Bruhat--Tits building $\,\BT_d$.  We denote the membrane by $\, [M]$, and we identify it
with the simplicial complex obtained by restricting $\BT_d$ to $[M]$.  
If $n=d$, so that
$M$ is a basis of $K^d$, then the membrane $[M]$ is known as an \emph{apartment} of the building
$\BT_d$. 

\begin{lemma}{\rm (Keel and Tevelev~\cite[Lemma 4.13]{KT})}\label{lem:membrane}\ \
  The membrane $[M]$ is the union of the apartments which can be formed from any $d$ linearly independent columns of $M$.
\end{lemma}

For instance, if we take $M = \{a,b,c,d,e,f,g,h\} \subset K^2$ as in Example~\ref{4leaftree},
then the membrane $[M]$ is an infinite tree with seven unbounded rays,
as shown in Figure \ref{Fig2} and derived in
Example \ref{EightRays} below. The convex hull of
$\Lambda_1 = R\{a,b\}$, $\Lambda_2 = R\{c,d\}$, $\Lambda_3 = R\{e,f\}$ and
$\Lambda_4 = R\{g,h\}$ was constructed as a finite subcomplex of the infinite tree $[M]$.

\begin{figure} \label{Fig2}\centering
  \includegraphics[scale=0.85]{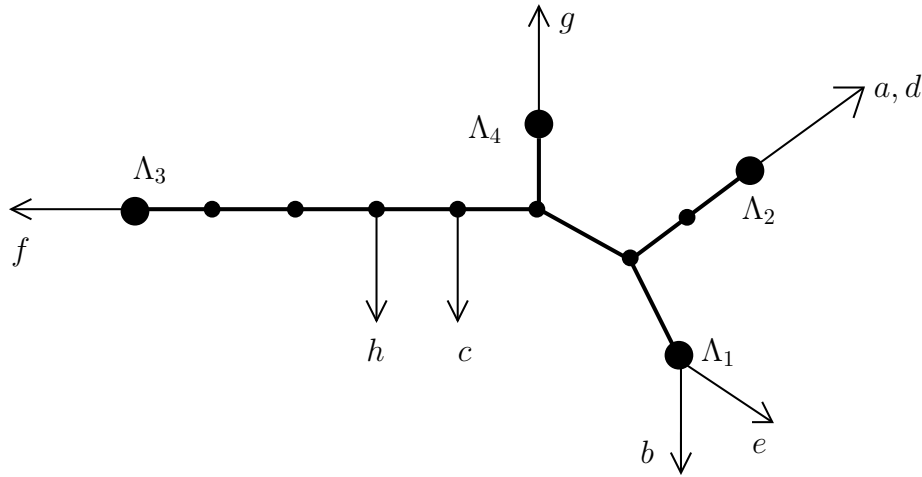}
  \caption{A one-dimensional membrane is an infinite tree.}
\end{figure}

The term ``membrane'' was coined by Keel and Tevelev \cite{KT} who showed that $[M]$ 
is a triangulation of the
tropicalization of the subspace of $K^n$ spanned by the rows of the $d \times n$-matrix $[f_1, \ldots,f_n]$.  This result is implicit in the work of Dress and Terhalle \cite{DT1, DT2}.  The precise statement 
and a self-contained proof will be given in Theorem~\ref{thm:membrane} below.

The membrane $[M]$ is obviously
max-convex in $\BT_d$.
However,  for $d \geq 3$,
membranes are generally not min-convex.
Here is a simple example which shows this:

\begin{example}
\label{ex:minconv}
We consider the $3 \times 5$-matrix
\[
M \,\, = \,\, (f_1,f_2,f_3,f_4,f_5) \,\, = \,\,
\left(
\begin{array}{ccccc}
z & 0 & 0& 1 & 1\\
0 & 1 & 0 & 1 & 0\\
0 & 0 & 1 & 0 & 1
\end{array}
\right)
\]
The lattices $\,\Lambda_1 = R\{f_1,f_2,f_3\}\,$ and $\,\Lambda_2 = R \{f_1, f_4, f_5\}\,$
are in the membrane $[M]$.
However, their intersection $\,\Lambda_1 \cap \Lambda_2 \,=\, R (0,1,-1) + z R^3\,$ 
is a lattice which is not in $[M]$. \qed
\end{example}

While apartments and membranes are infinite subcomplexes of the Bruhat--Tits building $\BT_d$, they have a natural
finite presentation by matrices whose columns are in $K^d$.  We can thus ask computational questions about apartments
and membranes, such as:

\medskip

\noindent \textbf{Computational Problem B}.
Compute the intersection of $s$ given apartments (or membranes) in $\BT_d$.  The input is represented by rank $d$
matrices $M_1,\ldots,M_s$ having $d$ rows with entries in $ K$.  The $i$-th apartment (or membrane) is
spanned by the columns of $M_i$.  The desired intersection is a locally finite simplicial complex of dimension $\leq
d-1$.

\medskip

General solutions to Problems A and B, based on tropical convexity, will be presented in Sections \ref{sec:algorithms} and
\ref{sec:perspectives}.  At this point, the reader may wish to contemplate our two problems for the special
case $d=2$: the intersection of apartments is a path which is usually finite.

\begin{remark}\label{rem:convex-def}
  In the theory of buildings there is another frequently used notion of convexity.
  Following \cite[\S3.6.2]{Bro},
  it rests on the following definitions.   The maximal
  simplices in the Bruhat--Tits building $\BT_d$ are called \emph{chambers}.  A set $\cC$ of chambers is 
  \emph{convex} if every chamber on a shortest path 
  (in the dual graph of the simplicial complex $\BT_d$)
  between two chambers of $\cC$
  also lies in~$\cC$.  This notion of convexity on $\BT_d$ agrees with
  convexity induced by shortest geodesics on spaces of non-positive curvature,
  and it is related to decompositions of semi-simple Lie groups \cite{Hit}.
  Apartments and sub-buildings as well as intersections of convex sets are convex.  A set of chambers contained in an
  apartment is convex if and only if it is the intersection of \emph{roots} (or \emph{half-apartments}).  In a thick
  building, such as $\BT_d$, every root is the intersection of two apartments.  Hence any convex set within some
  apartment of $\BT_d$ arises as the output of an algorithm for Computational Problem B.  Such algorithms are our topic
  in Section~\ref{sec:perspectives}.  The relationship of this \emph{classical convexity}  in $\BT_d$ to min- and max-convexity will be clarified in
  Proposition~\ref{prop:convex} and Theorem~\ref{thm:alessandrini}.
\end{remark}

\section{Tropical polytopes}\label{sec:proj}

\noindent
We review the basics of tropical convexity from \cite{DS}.  A subset $P$ of $\R^d$ is called \emph{tropically convex} if
it is closed under linear combinations in the min-plus algebra, i.e.~for any two vectors $x = (x_1,\dots,x_d)$ and $y =
(y_1,\dots,y_d)$ in $ P$ and any scalars $\lambda,\mu \in \R$ we also have
\[
\bigl(\min(x_1 + \lambda, y_1 + \mu), \ldots, \min(x_d + \lambda, y_d + \mu) \bigr) \,\,\, \in\,\,P.
\]
It has become customary to write the \emph{tropical arithmetic} operations as
\[
x\oplus y \ := \ \min(x,y) \quad \text{and} \quad x\odot y \ := \ x+y \, .
\]
In particular, if $x=(x_1,\dots,x_d) \in P$ then $\lambda\odot x := (\lambda \odot x_1, \dots,\lambda \odot x_d) \in P$
for all $\lambda$. Thus we can identify each tropically convex set $P \subset \R^d$ with its image in the \emph{tropical
  projective space}, which is defined as the quotient space
\[
\TP^{d-1} \quad := \quad \R^d / \R (1,1,\dots,1) \, .
\]
There is a natural metric $\delta$ on tropical projective space $\TP^{d-1}$ which is given as follows:
\begin{equation}\label{metric}
  \delta(x,y) \quad := \quad \max_{1 \leq i < j \leq d} | x_i + y_j - x_j - y_i |. 
\end{equation}
The following characterizes the projection to the nearest point in
a closed convex set.

\begin{prop}\label{prop:pi}
  Let $x \in \TP^{d-1}$ and $P$ a closed tropically convex
  set in $\TP^{d-1}$.  Among all points $w $ in $P$ that satisfy $\,
  w \geq x \,$ there is a unique coordinate-wise minimal point.
  (Here ``$w \geq x$'' means that there exist representative vectors $w,x \in \R^d$ with $w_i \geq
  x_i$ for all $i$).
    This point, which is denoted $\pi_P(x)$, minimizes the $\delta$-distance
  from $x$ to $P$.  
\end{prop}

\begin{proof}
  If $w,w' \in P$ then the coordinate-wise minimum $\min(w,w')$ also lies in $P$. Since $P$ is closed, it follows that
  the set $\smallSetOf{ w \in P }{ w \geq x }$ has a minimal point $y$.  We claim that $y$ is $\delta$-closest to $x$
  among all points in $P$.  Consider any point $y' \in P$.  After translation we may assume $x = 0$ and that both $y$
  and $y'$ have its smallest coordinate zero.  Then $\delta(x,y)$ is the largest coordinate of $y$, and $\delta(x,y')$
  is the largest coordinate of $y'$.  By construction of $y = \pi_P(x)$, we have $y_i \leq y'_i$ for all $i$, and hence
  $\delta( x,y) \leq \delta(x,y')$.
\end{proof}

The map $\pi_P \,: \,\TP^{d-1} \rightarrow P ,\, x \mapsto \pi_P(x)$ is the
\emph{nearest point map} onto $P$.  Clearly, $\pi_P(x) = x$ if and only if $x \in P$.  We now give an explicit formula
for $\pi_P$ in the special case when $P$ is a \emph{tropical polytope}.  This means that $P$ is the smallest
tropically convex set containing a given finite collection of points $v_1,v_2,\dots,v_n \in \T \PP^{d-1} $.  Thus $P$
is the \emph{tropical convex hull} of these points, in symbols, $P = \tconv(v_1,v_2,\ldots,v_n)$.

\begin{lemma}\label{lem:pi}
The $i$-th coordinate of 
the nearest point map onto the tropical polytope $P = \tconv(v_1,v_2,\ldots,v_n)$
in $\TP^{d-1}$ is given by the formula
\[
\pi_P(x)_i \,\, = \,\, \min_{k\in\{1,\ldots,n\}} \max_{j \in \{1,\ldots,d\}} (v_{ki} -v_{kj} + x_j).
\]
\end{lemma}

\begin{proof}
  Set $y _i = \min_{k=1}^n \max_{j=1}^d (v_{ki} -v_{kj} + x_j)$.  
  Taking $j=i$ in the maximum, we 
  see that the vector $y = (y_1,\dots,y_d)$ satisfies $y \geq x$. Writing $y _i = \min_{k=1}^n (\max_{j=1}^d
  (x_j\!-\!v_{kj}) + v_{ki})$, we find that $y$ is a tropical linear combination of the points $v_1,\dots,v_n$.  Hence
  $y$ lies in $P$.  Moreover, $y$ is the coordinate-wise minimal vector
  in $\R^d$ with these two properties.
\end{proof}

\begin{example}
  There may be several points in a tropical polytope $P$ which minimize the distance to 
  a given point $x$.   
  Consider the point  $x=(0,1,1) $ in the plane $\TP^2$ and the  
  one-dimensional polytope $P=\tconv((1,0,0),(0,1,0),(0,0,1))$.
  The projection
  of $x$ onto~$P$ is $\pi_P(x)=(0,0,0)=(1,1,1)$, but
  $\,\delta(x,(0,0,0))\,=\,\delta(x,(0,0,1))\,=\,\delta(x,(0,1,0))\,=\,1$.  \qed
\end{example}

The formula in Lemma~\ref{lem:pi} specifies a subdivision of the tropical polytope $P$ into cells.
These \emph{cells} are
ordinary polytopes of the special form
\begin{equation}\label{cell} \quad
  \SetOf{ w \in \TP^{d-1} }{ w_i - w_j  \leq u_{ij} \text{ for all $i \neq j$ } }
  \qquad
  (\text{for some } u_{ij} \in \R). \!\!\!
\end{equation}
The cell containing $x \in P$ is specified by its \emph{type}, which is the 
collection of index sets where ``$\min$'' and
``$\max$'' are attained in the identity $\pi_{P}(x) = x$. 
To be precise, we define $\type(x) := (S_1,S_2,\ldots,S_d)$,
where 
\begin{equation}\label{eq:type}
  \begin{split}
    S_i \ &= \ \SetOf{k \in \{1,\ldots,n\} }{ \max_{j\in\{1,\dots,d\}} ( v_{ki} - v_{kj} + x_j) \, = \, x_i } \\
    &= \ \SetOf{ k }{ v_{ki} - x_i \, = \, \min(v_{k1}-x_1, v_{k2}-x_2,\ldots, v_{kd}-x_d) } \, .
  \end{split}
\end{equation}
Two points of $P$ lie in the same cell if and only if they have
the same type.  This subdivision of $P$ depends
on the chosen generators $v_1,v_2,\ldots,v_n$ and not just on the set $P$.

\begin{remark}\label{rem:root}
  The sets
  $\,
  \SetOf{w \in \TP^{d-1}}{w_i  - w_j  \,\leq \,u}
  \,$
  are the ordinary affine halfspaces which are also
  tropically convex.  For integral $u$ we call such a halfspace a \emph{root} of $\TP^{d-1}$.
\end{remark}

A point in the tropical projective space $\TP^{d-1}$ is a \emph{lattice point} if it is represented by a
vector $x$ in $\Z^d$.  We define a graph on the set of all lattice points as follows: two points $x$ and $y$ are
connected by an edge if and only if $\delta(x,y) = 1$.  The $\delta$-distance between any two lattice points in
$\TP^{d-1}$ is the shortest length of any path connecting these two points in the graph.  A \emph{tropical lattice
  polytope} is the tropical convex hull of finitely many lattice points in $\TP^{d-1}$.  
  The cells of a tropical lattice polytope are intersections of roots.

\begin{theorem} \label{thm:triang}
  The flag simplicial complex defined by this graph is a triangulation of the affine space $\TP^{d-1}$.  It restricts to
  a triangulation of each cell \eqref{cell} of each tropical lattice polytope $P$. We refer to this 
  as the \emph{standard triangulation of }$\TP^{d-1}$, or of $P$, or of \eqref{cell}.
  \end{theorem}

\begin{proof}
  We represent points in $\TP^{d-1}$ by vectors with first coordinate zero. This identifies the lattice points in
  $\TP^{d-1}$ with $\Z^{d-1}$.  The maximal simplices in the flag complex are
  \[
  \bigl\{ a , a + e_{\sigma_2}, a + e_{\sigma_2} + e_{\sigma_3} , \dots,
  a + e_{\sigma_2} + e_{\sigma_3} + \dots + e_{\sigma_d} \bigr\} \, ,
  \]
  where $u \in \Z^{d-1}$ and $\sigma$ is any permutation of $\{2,\dots,d\}$. If we fix $a$ and let $\sigma$ range over all $(d-1)!$
   permutations then these simplices triangulate the unit cube with lower vertex $a$.  Putting all these triangulated
  cubes together, we see that the flag complex is a triangulation of $\TP^{d-1}$.
      Each simplex in this standard triangulation is the solution set to a system of inequalities $ w_i - w_j \leq u_{ij}$
  where $u_{ij} + u_{ji} \leq 1$ for all $1 \leq i < j \leq d $.  This implies that if $w$ is any point in a cell
  \eqref{cell} then that cell contains the entire simplex of the standard triangulation which has $w$ in its relative
  interior.  Therefore the standard triangulation of $\TP^{d-1}$ induces a triangulation of every tropical lattice
  polytope.
\end{proof}

\begin{figure}\label{fig:polygon}\centering
  \includegraphics[width=.9\linewidth]{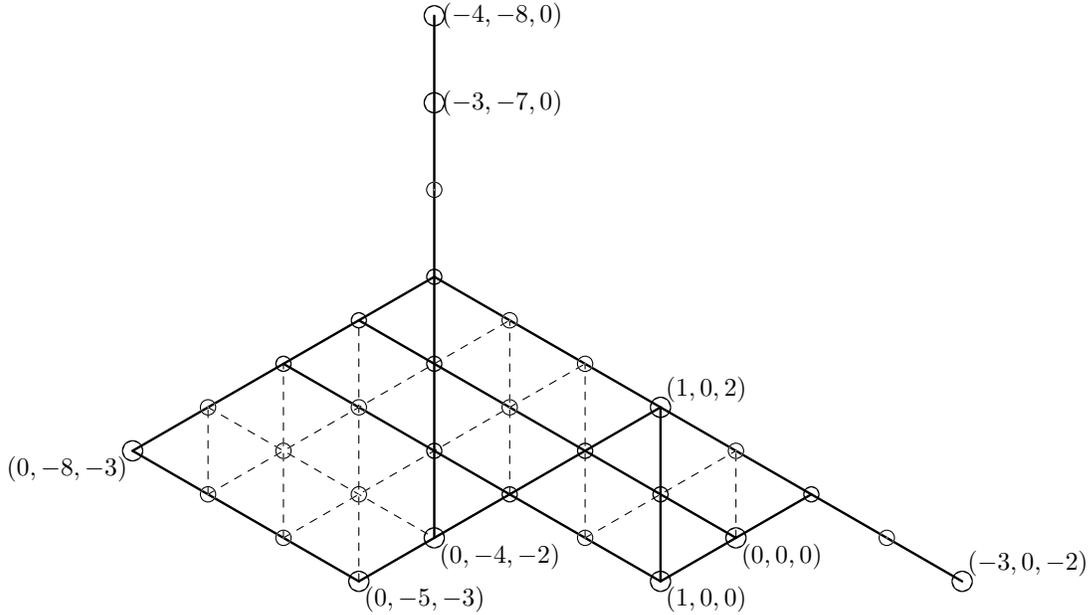}
  \caption{The tropical convex hull of nine labeled lattice points in $\TP^2$.  Dashed lines and white points indicate
    the standard triangulation of this polygon.  Solid lines and black points show the decomposition into cells
    (\ref{cell}).}
\end{figure}

\begin{example}[$d=3, n=9$] 
\label{ThreeNine}
Let $v_1,v_2,\ldots,v_9$ denote the columns of  
\begin{equation}
\label{39matrix}
 V \,\,\, = \,\,\,
\begin{pmatrix}
  \phantom{-}0 & \phantom{-}0 & \phantom{-}0 &    1 & -3 & 1 &    -3 & -4 & 0 \,\\
  -5 & -4 & -8 &  0 & \phantom{-}0 & 0  & -7 & -8 & 0 \, \\
  -3 & \phantom{-}2 & -3 & 0 & -2 & 2 & \phantom{-}0 & \phantom{-}0 & 0\, 
\end{pmatrix}
\end{equation}
We compute the tropical convex hull $\, P = \tconv(v_1,\ldots,v_9)\,$ in $\TP^2$.  The tropical lattice polygon $P$ has
ten $2$-dimensional cells, $28$ edges, and $19$ vertices.  Hence there are $10+28+19=57$ distinct types $\,\type(x) =
(S_1,S_2,S_3)\,$ among the points $x$ in $P$. The standard triangulation of $P$ is a simplicial complex with $32$
triangles, $62$ edges and $31$ vertices, namely, the lattice points in $P$. It is depicted in Figure \ref{fig:polygon}.
\qed
\end{example}

By \cite[Theorem 23]{DS},
the convex hull of the rows of a matrix equals the convex hull of the columns of
that same matrix.
Indeed, if $V$ is the $d \times n$-matrix whose columns are the vectors
$v_i$ then the cell complex on $P = \tconv(v_1,\ldots,v_n)$ defined by the types is isomorphic to the cell complex on the convex hull in $\TP^{n-1}$ of the $d$ row vectors of $V$.

\begin{example} \label{ThreeNineAgain}(Self-Duality of Tropical Polytopes) \ 
  Let $v'_1,v'_2,v'_3$ be the row vectors of the matrix $V$ in \eqref{39matrix}, and let $P' = \tconv(v'_1,v'_2,v'_3)$
  be their tropical convex hull in $\TP^8$.  The tropical triangle $P'$ contains precisely the following $31$ lattice
  points:

  \smallskip

  \begin{tabular*}{.9\linewidth}{@{\extracolsep{\fill}}ccc@{}}
    ${\bf (\underbar{\bf 4},\underbar{\bf 4},\underbar{\bf 4},5,1,5,1,0,4)}$ &
    $(\underbar{4},\underbar{4},\underbar{3},5,1,5,1,0,4)$ &
    $(\underbar{4},\underbar{4},\underbar{2},5,1,5,1,0,4)$\\
    $(\underbar{4},\underbar{4},\underbar{1},5,1,5,1,0,4)$ &
    $(\underbar{3},\underbar{4},\underbar{3},5,1,5,1,0,4)$ &
    $(\underbar{3},\underbar{4},\underbar{2},5,1,5,1,0,4)$\\
    $(\underbar{3},\underbar{4},\underbar{1},5,1,5,1,0,4)$ &
    $(\underbar{3},\underbar{4},\underbar{0},5,1,5,1,0,4)$ &
    $(\underbar{2},\underbar{4},\underbar{2},5,1,5,1,0,4)$\\
    $(\underbar{2},\underbar{4},\underbar{1},5,1,5,1,0,4)$ &
    $(\underbar{2},\underbar{4},\underbar{0},5,1,5,1,0,4)$ &
   
    $(\underbar{3},\underbar{4},0,\underbar{6},2,6,1,0,5)$\\

    $(\underbar{3},4,0,\underbar{7},\underbar{3},7,1,0,6)$ &
    $(\underbar{3},4,0,\underbar{8},\underbar{4},8,1,0,7)$ &

    $(\underbar{1},\underbar{4},\underbar{1},4,1,5,1,0,4)$\\
    $(\underbar{1},\underbar{4},\underbar{0},4,1,5,1,0,4)$ &

    $(\underbar{2},\underbar{4},0,5,\underbar{2},6,1,0,5)$ &
    $(\underbar{3},\underbar{4},0,6,\underbar{3},7,1,0,6)$\\

    $(\underbar{3},4,0,\underbar{7},\underbar{4},8,1,0,7)$ &
    $(\underbar{3},4,0,\underbar{8},\underbar{5},8,1,0,8)$ &

    $(\underbar{0},\underbar{4},\underbar{0},3,1,5,1,0,3)$\\

    $(\underbar{1},\underbar{4},0,4,\underbar{2},6,1,0,4)$ &
    $(\underbar{2},\underbar{4},0,5,\underbar{3},7,1,0,5)$ &
    $(\underbar{3},\underbar{4},0,6,\underbar{4},8,1,0,6)$\\

    $(\underbar{3},4,0,\underbar{7},\underbar{5},8,1,0,7)$ &
    $(\underbar{3},4,0,\underbar{8},\underbar{6},8,1,0,8)$ &
    $(\underbar{3},4,0,\underbar{8},\underbar{7},8,1,0,8)$\\
    ${\bf (\underbar{\bf 3},4,0,\underbar{\bf 8},\underbar{\bf 8},8,1,0,8)}$ &

    $(\underbar{0},\underbar{5},\underbar{0},3,1,5,2,1,3)$ &

    $(\underbar{0},5,\underbar{0},3,1,5,\underbar{3},2,3)$\\

    ${\bf (\underbar{\bf 0},5,\underbar{\bf 0},3,1,5,3,\underbar{\bf 3},3)}$ &\\
  \end{tabular*}
\smallskip

\noindent
Here each point is represented uniquely by a non-negative vector with a zero entry. The boldfaced vectors represent the
given points $v'_1,v'_2,v'_3 \in \TP^8$.  The underlined triples of coordinates will be explained in Example
\ref{ex:ProbA}.  The tropical triangle $P'$, which lives in $\TP^8$, is isomorphic to the tropical $9$-gon $P$ of
Example \ref{ThreeNine}, which lives in $\TP^2$ and is depicted in Figure \ref{fig:polygon}. According to equation (14)
in \cite[page 16]{DS}, the isomorphism between the two tropical polygons is given by the piecewise-linear maps
\begin{equation}\label{eq:iso}
  \begin{split}
    P \rightarrow P',\
    &(x_1,x_2,x_3) \,\mapsto
    \bigl(\min_{i=1}^3 (v_{i1} - x_i), \ldots, \min_{i=1}^3 ( v_{i9} -
    x_i) \bigr) \, ,\\
    P' \rightarrow P,\
    &(y_1,\ldots,y_9) \mapsto
    \bigl(
    \min_j (v_{1j}{-} y_j),    \dots,
    \min_j (v_{3j} {-} y_j) \bigr) \, . 
  \end{split}
\end{equation} 
These bijections are inverses of each other. They are
linear on each cell, and
they identify the types: 
if $x \in P$ and $\type(x) = (S_1,S_2,S_3)$ 
then the corresponding point $y \in P'$ has 
$\type(y) = (S'_1,S'_2,\ldots,S'_9)$
where $\,S'_j = \{i: j \in S_i\}$.
The $31$ lattice points in $\TP^8$ that are listed above
get sent to the $31$ lattice points in Figure \ref{fig:polygon}
by the map $P' \rightarrow P$. \qed
\end{example}

We close with the remark that several algorithms are available for computing a tropical polytope $P$ from its defining
matrix $V = (v_{ij})$.  They will be discussed in Section~\ref{sec:algorithms}.

\section{Tropical linear spaces and membranes}\label{sec:membranes}

\noindent
This section is concerned with the relationship between tropical linear spaces, valuated matroids \cite{DT1, DT2}, and
membranes \cite{KT} in the Bruhat--Tits building.  In order to think of these objects as tropical polytopes, we shall now
augment the real numbers $\R$ by the extra element $\infty$. Note that $\infty$ is the additively neutral element in the
min-plus algebra.  We define the \emph{compactified tropical projective space} $\overline{\TP}^{d-1}$ to be $(\R \cup
\{\infty\})^d \setminus \{(\infty,\ldots,\infty)\}$ modulo the equivalence relation given by tropical scalar
multiplication.  The notions of tropical convexity, tropical polytopes and lattice points make sense in
$\overline{\TP}^{d-1}$.  When extending the metric $\delta$ to $\overline{\TP}^{d-1}$ we use the convention that $\infty
- \infty = 0$ in the formula (\ref{metric}).  Proposition~\ref{prop:pi} and Lemma~\ref{lem:pi} remain valid, and there
is a standard triangulation of $\overline{\TP}^{d-1}$.  That standard triangulation coincides with the compactified
apartment in the work of Werner \cite{We1,We2}.  We also refer to Alessandrini \cite{Ale} whose tropical approach to
buildings is similar to ours and is aimed at applications in Teichm\"uller theory.

For experts on buildings we note that our two notions of convexity in Problem A reflect two different compactifications
of the Bruhat--Tits buildings $\BT_d$. The first is featured in \cite{Mus, We1} and we call it the
\emph{max-compactification}.  It is a simplicial complex whose vertices are all free $R$-submodules of $K^d$, and the
boundary consists of modules of rank less than $d$.  The second compactification, which we call the
\emph{min-compactification}, arises more naturally from tropical geometry.  Its points consist of all additive seminorms
on $K^d$. An additive seminorm is a function $N : K^d \rightarrow \R \cup \{\infty\}$ which satisfies the first two
axioms of an additive norm. If $N$ is an additive seminorm then $N^{-1}(\infty)$ is a linear subspace of $K^d$.  The
boundary of the min-compactification consists of additive seminorms for which $N^{-1}(\infty)$ is positive-dimensional.
We shall not dwell on the matters here, but we do wish to underline that our combinatorial results are compatible with
these compactifications.

\smallskip

We now review the definition of tropical linear spaces \cite{Spe1, SS}.  Fix two positive integers $d \leq n$ and
consider a map $p : \{1,2,\dots,n\}^d \rightarrow \R \cup \{\infty\}$.  Following Dress and Terhalle \cite{DT1, DT2}, we
say that $p$ is a \emph{valuated matroid} if $p(\omega)$ depends only on the unordered set
$\{\omega_1,\dots,\omega_d\}$, and $p(\omega) = \infty$ whenever $\omega$ has fewer than $d$ elements, and $p$ satisfies
the following variant of the basis exchange axiom: for any $(d-1)$-subset $\sigma$ and any $(d+1)$-subset $\tau $ of
$\{1,2,\dots,n\}$, the minimum of the list of numbers $\bigl( p(\sigma \cup \tau_i) + p(\tau \setminus \{\tau_i\})\,:\,
i = 1,2,\ldots, d+1\,\bigr)$ is attained at least twice.  This axiom is equivalent to saying that $p$ lies in the
\emph{tropical prevariety} \cite{RGST} specified by the set of all \emph{quadratic Pl\"ucker relations}.

Fix a valuated matroid $p$. The associated \emph{tropical linear space} $L_p$ consists of all points $x \in
\overline{\TP}^{n-1}$ such that, for any $(d{+}1)$-subset $\tau$ of $\{1,2,\dots,n\}$, the minimum of the numbers $\,p(
\tau \backslash \{\tau_i\}) + x_{\tau_i}$, for $i=1,2,\ldots,d$, is attained at least twice.  This list of numbers
represents a \emph{circuit} of $p$.  The tropical linear space $L_p$ is tropically convex, and it can be represented as
a tropical lattice polytope as follows.  For any $(d{-}1)$-subset $\sigma$ of $\{1,\ldots,n\}$ let $p(\sigma *)$ denote
the vector in $(\R \cup \{\infty\})^n$ whose $j$-th coordinate equals $\, p(\sigma \cup \{j\}) $. We regard $p(\sigma
*)$ as a point in $\overline{\TP}^{n-1}$, or, combinatorially, as a \emph{cocircuit} of the valuated matroid $p$.

\begin{theorem}\label{thm:tropical_linear_space}
{\rm (Yu and Yuster \cite[Theorem~16]{YY}) \  }
  The tropical linear space $L_p$ is the tropical convex hull in the compactified tropical projective space
  $\overline{\TP}^{n-1}$ of all the cocircuits $\,p(\sigma *)$ of the underlying valuated matroid $p :
  \{1,2,\dots,n\}^d \rightarrow \R \cup \{\infty\}$.
\end{theorem}

The tropical linear space $L_p$ is tropically convex.  Hence it has a nearest point map $\pi_{L_p}$ which takes any
point $x \in \overline{\TP}^{n-1}$ to the coordinate-wise minimum in $\smallSetOf{w \in L_p}{w \geq x}$.  We now present
two rules for evaluating this map.

\smallskip

\noindent \textbf{The Blue Rule.}  Form the vector $w \in \R^n$ whose coordinates are
\begin{equation}\label{minmaxred}
  w_i \quad = \quad \min_\sigma \, \max_{j \not\in \sigma} 
  \bigl( p (\sigma \cup \{i\}) - p(\sigma \cup \{j\}) + x_j \bigr). 
\end{equation}
Here the minimum is over all $(d-1)$-subsets $\sigma$ of $\{1,2,\dots,n\}$.

\smallskip

\noindent \textbf{The Red Rule.}  Start with $v = (0,0,\ldots,0)$.  For any $(d+1)$-set $\tau$ do: If the minimum of the
numbers $p( \tau \backslash \{\tau_i\}) + x_{\tau_i}$ is attained only once, for the index $i$, then let
$\gamma_{\tau,i}$ be the difference of the second smallest number minus that minimum, and set $v_{\tau_i} :=
\max(v_{\tau_i}, \gamma_{\tau,i})$.

\smallskip

The terms \emph{Blue Rule} and \emph{Red Rule} were introduced by Ardila~\cite{Ard}. 
The following theorem extends his main result in \cite{Ard} from ordinary
matroids to valuated matroids:

\begin{theorem}\label{thm:blue-red}
  Let $p$ be a valuated matroid, $L_p$ its tropical linear space and $x \in \overline{\TP}^{n-1}$. If $v$
  and $w$ are computed by the Red Rule and the Blue Rule then $\pi_{L_p}(x)
  \, = \, x+v \,= \, w $.
\end{theorem}

\begin{proof}[Sketch of Proof]
  In the case of ordinary matroids, the image of $p$ lies in $\{0,\infty\}$.  This special case is
  \cite[Theorem~1]{Ard}.  Ardila's proof easily generalizes to valuated matroids.  The correctness of the Blue Rule also
  follows from Lemma~\ref{lem:pi}~and Theorem~\ref{thm:tropical_linear_space}.  
\end{proof}

\begin{remark}\label{rem:0}
  The Red Rule and the Blue Rule produce the identical result in the special case when $x=(0,0,\ldots,0)$.
  We find that  $\,\pi_P(0,0,\dots,0) \in L_p\,$ is the tropical sum of all cocircuits $p(\sigma *)$ of the valuated matroid $p$, provided each cocircuit is represented by the unique vector whose
 coordinates are non-negative and has at least one coordinate zero.
\end{remark}
  
We now apply tropical convexity to the Bruhat--Tits building $\BT_d$.  We begin with a review on how
tropical linear spaces are related to ordinary linear spaces over the field $K = \C(\!(x)\!)$.  Let $M$ be a $d \times
n$-matrix of rank $d$ with entries in $K$. The row space of $M$ is a $d$-dimensional linear subspace of $K^n$, or a
$(d-1)$-dimensional subspace of the projective space $\PP^{n-1}_K$. If $\omega$ is an ordered list of $d$ elements
in $\{1,2,\dots,n\}$ then $M_\omega$ denotes the corresponding $d \times d$-submatrix.  The matrix $M$ defines a
valuated matroid $p$ by the rule
\begin{equation} \label{valuatedM}
p(\omega) \quad = \quad \val \bigl( \det(M_\omega) \bigr) \, .
\end{equation}
Note that $p(\omega) = \infty$ if and only if $M_\omega$ is not invertible over $K$.

\begin{prop} \label{prop:linear} {\rm (Speyer and Sturmfels \cite[Theorem 2.1]{SS})}
The lattice points in the tropical linear space $L_p$ are precisely the points $\val(v)$ where $v$ 
is in the row space of~$M$.
\end{prop}

Since $L_p$ is a tropical lattice polytope, the standard triangulation of $\overline{\TP}^{n-1}$ restricts to a
triangulation of $L_p$.  We shall present a self-contained proof of the following result.

\begin{theorem}\label{thm:membrane} {\rm (Keel and Tevelev \cite[Theorem 4.15]{KT})} Let $ \,M = (f_1,
  f_2,\ldots,f_n)\,$ be a $d \times n$-matrix of rank $d$ over $K$, and let $L_p $ be the associated tropical linear
  space. Then
  \[
  \Psi_M \,:\,    R \bigl\{ z^{-u_1} f_1,\, 
  z^{-u_2} f_2,\, \ldots ,\,z^{-u_n} f_n \bigr\}  \, \mapsto \, \pi_{L_p}(u_1,u_2,\dots,u_n)
  \]
  is a well-defined map, and it induces an isomorphism of simplicial complexes between the membrane $[M]$ and the
  standard triangulation of $L_p$.
\end{theorem}

\begin{proof}
  Consider any lattice $\,\Lambda =
  R \bigl\{ z^{-u_1} f_1,\,   z^{-u_2} f_2,\, \ldots ,\,z^{-u_n} f_n \bigr\}  \,$
   in the membrane, and set
  $(v_1,v_2,\ldots,v_n) = \pi_{L_p} (u_1,u_2,\ldots,u_n)$.  We claim that
  \begin{equation}\label{redmodule}
    v_i \quad = \quad \max \{ \, \mu \in \Z \,: \, z^{-\mu} f_i \in \Lambda \,\} . 
 \end{equation}
 We first prove the inequality ``$\leq $''. By the Red Rule in Theorem~\ref{thm:blue-red}, we have $v_i =
 \gamma_{\tau,i} + u_i$ for some $(d+1)$-set $\tau$ containing $i$. We may assume $\tau_{d+1} = i$. 
 Then
 $\{f_{\tau_1}, \dots,f_{\tau_d}\}$ is a basis of $K^d$, and we can write 
 \[
 \quad f_i \quad = \quad p_1 f_{\tau_1} + p_2 f_{\tau_2} + \cdots + p_d f_{\tau_d} \qquad \quad 
 \hbox{for some} \,\, p_1,\ldots,p_d \in K.
 \]
 Our choice of the $(d+1)$-set $\tau$ in the Red Rule means that
 \[
 u_i + \gamma_{\tau,i} \quad = \quad \min \{\,\val(p_j) + u_{\tau_j} \,:\, j = 1,2,\dots,d \,\} \quad \geq \quad 0 \, ,
 \]
 and therefore
 \begin{equation} \label{frep}
 f_i z^{-u_i-\gamma_{\tau,i}} \,\,\, = \,\,\, p_1 z^{u_{\tau_1}} ( f_{\tau_1}z^{-u_{\tau_1}}) + \cdots + p_d
 z^{u_{\tau_d}} ( f_{\tau_d}z^{-u_{\tau_d}}) \,\, \in \,\, \Lambda \, .
 \end{equation}
 This proves the inequality ``$\leq $''.  The converse ``$\geq$'' holds because $z^{-\mu} f_i$ lies in~$\Lambda$ if and
 only it lies in the $R$-submodule spanned by $d$ of the $n$ generators, and a representation (\ref{frep}) is the only
 way this can happen. Indeed, by Lemma~\ref{lem:membrane}, the membrane $[M]$ is the union of the apartments
 $[(f_{\tau_1},\dots,f_{\tau_d})]$ for all $d$-subsets $\tau\subseteq\{1,2,\dots,n\}$.
 
 The identity (\ref{redmodule}) shows that the map $\Psi_M$ which takes the lattice 
 $ R \bigl\{ z^{-u_1} f_1,\dots,\,z^{-u_n} f_n \bigr\}$
 to the point $\pi_{L_p}(u_1,\dots,u_n)$ is well-defined, and is a bijection
 between  the membrane $[M]$ and the lattice points in the tropical linear space $L_p$. This bijection takes
 adjacent lattices to points of $\delta$-distance one in $L_p$ and conversely. Hence it induces an isomorphism between
 the flag simplicial complexes of these two graphs.
\end{proof}

\begin{example} \label{EightRays}
Let $d=2$, $n=8$ and let $\, M = (a,b,c,d,e,f,g,h)\,$
 be as in Example \ref{4leaftree}.
The valuated matroid $\,p\,$ of the matrix $\,M\,$ maps
pairs of columns to $\Z \cup \{\infty\} $ as follows:
\[ 
\begin{pmatrix}
aa & \! ab & \! ac & \cdots & ah \\
ab & \! bb & \! bc & \cdots & bh \\
ac & \! bc & \! cc & \cdots & ch \\
\vdots &  \vdots & \vdots & \ddots & \vdots \\
ah & \! bh & \! ch & \cdots & hh 
\end{pmatrix}
\mapsto
\begin{pmatrix}
 \infty & -7 & -2 &  \infty & -1 & -7 & -3 & -2 \\
                    -7 &  \infty & -3 & -5 & -1 & -8 & -4 & -3 \\
                      -2 & -3 &  \infty & 0 & 3 & -1 & 2 & 4 \\
                    \infty & -5 & 0 &  \infty & 1 & -5 & -1 & 0 \\
                      -1 & -1 & 3 &  1 &  \infty & -2 & 2 & 3 \\
                     -7 & -8 & -1 & -5 & -2 &  \infty & -3 & 0 \\
                      -3 & -4 & 2 & -1 & 2 & -3 &  \infty & 2 \\
     -2 & -3 &  4 &  0 &  3 &  0 &  2 &  \infty 
\end{pmatrix}
\]
The rows of this $8 {\times} 8$-matrix are the cocircuits $p(\sigma *)$ of the valuated matroid $p$.  They
represent seven distinct points in $\overline{\TP}^{7}\,$ (rows $1$ and $4$ give the same point).
The tropical linear space $L_p$ is the tropical convex hull of these seven points in $\overline{\TP}^{7}$.  This convex hull is the
tree depicted in Figure \ref{Fig2}.  A systematic algorithm for drawing such a tree $L_p$, given its valuated matroid
$p$, is the \emph{neighbor-joining method} from phylogenetics; see \cite[\S 6]{SS}. \qed
\end{example}

Theorem \ref{thm:membrane} states that every lattice point $(u_1,\ldots,u_n) $ in $ L_p$ uniquely represents a lattice
$\Lambda_u = R \bigl\{ z^{-u_1} f_1,\, \ldots ,\,z^{-u_n} f_n \bigr\}$ in the membrane $[M]$. The lattice $\Lambda_u$
specifies a matroid $M_u$ of rank $d$ on $\{1,2,\ldots,n\}$.  This is an ordinary (not valuated) matroid. The bases of
$M_u$ are the sets $\{\tau_1,\ldots,\tau_d\}$ such that $\,\{ z^{-u_{\tau_1}} f_{\tau_1},\, \ldots ,\,z^{-u_{\tau_d}}
f_{\tau_d} \}\,$ spans the lattice $\Lambda$. The matroid $M_u$ can be read off directly from the valuated matroid $p$ as follows: its bases are the $d$-sets $\tau$
such that the expression $\,p(\tau) - u_{\tau_1} - \cdots - u_{\tau_d}\,$ is minimal.  The set of all matroids $M_u$, as
$u$ ranges over the tropical linear space $L_p$, forms a \emph{matroid subdivision} of the matroid polytope of the
matrix $M$ over the field $K$.  This is the identification of tropical linear spaces with matroid subdivisions as
studied in \cite{Ka, Spe1}.

Our algorithm for Computational Problem A in Section 5 will output each lattice $\Lambda_u$ in
the min-convex hull as a pair $(u,M_u)$, where $u$ is a point in a 
tropical linear space $L_p$ and $M_u$ is
a matroid.  We saw this format already in Example \ref{4leaftree}. For instance, consider the point $u=(2, 0, 5,
4, 6, 0, 4, 5)$ listed there.  It lies the tropical line $L_p$ of Example \ref{EightRays}.  The rank $2$
matroid $M_u$ has the set of bases $\{ab, ac, ae, af, ag, ah, bd, cd, de, df, dg, dh\}$.

\smallskip

The classical notion of convexity in buildings in Remark~\ref{rem:convex-def}
is related to  tropical convexity as follows.
  For a chamber $C$ in $\BT_d$ let $\vertices(C)$ be its set of vertices.  Now
consider a set $\cC$ of chambers contained in some apartment $\cA$.  We identify $\cA$ with $\TP^{d-1}$ and we note that the
classical notion of a \emph{root} (or
\emph{half-apartment}) of $\cA$ agrees with our definition of a root in $\TP^{d-1}$ from Remark~\ref{rem:root}. We consider the following set of 
lattice points in $\TP^{d-1}$:
\[
\vertices(\cC) \ := \ \bigcup\SetOf{\vertices(C)}{C\in\cC}
\]
Our next result holds because the convex subsets of chambers in $\cA$ are intersections of
roots, or equivalently, intersections of $\cA$ with other apartments.
 See also Theorem~\ref{thm:alessandrini}.

\begin{prop}\label{prop:convex}
  A finite set $\cC$ of chambers in an apartment $\cA\cong\TP^{d-1}$ is convex if and only if $\vertices(\cC)$ is the set of lattice points in a tropical lattice polytope of the form (\ref{cell}).
  \end{prop}

Proposition \ref{prop:convex} implies that
the convex sets of chambers are precisely the maximal simplices in the standard triangulation of
those tropical lattice polytopes which are at the same time (possibly unbounded) ordinary convex polyhedra.  In other words,
 Proposition \ref{prop:convex} holds verbatim for infinite $\,\cC\,$ if $\,\TP^{d-1}\,$ is replaced by its
compactification $\,\overline{\TP}^{d-1}$.

\section{Convex hulls in the Bruhat--Tits building}\label{sec:algorithms}

\noindent
In this section and the next we present algorithmic implications of the theory developed 
so far. We begin with Computational Problem A: how to find min-convex hulls in $\BT_d$. 
The input is a
list of $s$ invertible $d \times d$-matrices $M_1,M_2,\dots,M_s$ with entries in the field $K = \C(\!(z)\!)$, each
representing the equivalence class of its column lattice $\Lambda_i = \image_R(M_i)$.

\subsection{The retraction of min-convex hulls to a membrane}

Let $M = (f_1, \dots, f_n) $ be any matrix in $ K^{d \times n}$ of rank $d$ and
let $[M]$ be the membrane in $\BT_d$ which is spanned by the $n$ column vectors of $M$.
There is a natural retraction $r_M$ from $\BT_d$ onto $[M]$ given by
\begin{equation}
\label{retraction}
r_M : \Lambda \mapsto (\Lambda \cap K\{f_1\}) + \cdots +  (\Lambda \cap K\{f_n\})
\end{equation}
This map restricts to the identity on the membrane $[M]$.

Let $V$ be the $d$-dimensional subspace of $K^{n}$ spanned by the rows of $M$, and let $p$ be its valuated matroid as in
formula (\ref{valuatedM}).  By Proposition \ref{prop:linear}, the tropicalization of the classical linear space $V$ over
the field $K$ equals the tropical linear space $L_p $.  The map $\Psi_M$ in Theorem~\ref{thm:membrane} allows us to
identify the lattice points in $L_p$ with the membrane $[M]$.

\begin{lemma}
\label{lem:3equiv}
Fix a membrane $[M]$ in $\BT_d$ and consider any lattice $\Lambda = \image_R(M_0)$ where $M_0 \in GL_K(d)$. Then the
following three lattice points in $\overline{\TP}^{n-1}$ coincide:
  \begin{enumerate}
  \item\label{item:membrane:1} $\Psi_M(r_M(\Lambda))$, where $\Psi_M$ is the bijection of Theorem~\ref{thm:membrane}
    between $[M]$ and the lattice points in $L_p$,
 \item\label{item:membrane:2} $(N_{\Lambda}(f_1), \dots, N_{\Lambda}(f_n))$, where $N_{\Lambda}$ is the integral
   additive norm corresponding to $\Lambda$,
  \item\label{item:membrane:3} the tropical sum 
  (coordinatewise minimum) of the rows of the matrix $\,\val(M_0^{-1}\cdot M)$.
  \end{enumerate}
\end{lemma}

\begin{proof}
  The equivalence of \eqref{item:membrane:1} and \eqref{item:membrane:2} follow from the definitions of $N_\Lambda$ and $r_M$, and from equation (\ref{redmodule}).  The equivalence of
  \eqref{item:membrane:2} and \eqref{item:membrane:3} follows from equation \eqref{norm}.
  \end{proof}
  
As a consequence, we get the following explicit description of the retraction of a min-convex hull onto a membrane. This establishes the correctness of Algorithm \ref{algo:retract} below.

\begin{prop}
\label{prop:Algo1}
Let $\Lambda_1,\Lambda_2, \ldots,\Lambda_s$ be the lattices spanned by the columns of 
the matrices $M_1, M_2, \dots, M_s \in GL_d(K)$. Let $[M]$ be any membrane in $\BT_d$.
The simplicial complex
\[r_M(\minconv(\Lambda_1,\Lambda_2, \ldots,\Lambda_s)) \subset [M] \,\]
coincides with the standard triangulation of the tropical polytope
\[
\tconv \bigl(\Psi_{M}(r_M(\Lambda_1)),\Psi_{M}(r_M(\Lambda_2)),\dots,\Psi_{M}(r_M(\Lambda_s)) \bigr) \,\,\, \subset \, \, L_p \,=\,\val(\kernel(M)).
\]
\end{prop}

\begin{proof}
By the definition of the integral additive norm $N_{\Lambda}$ in formula \eqref{norm}, we have
\[
N_{(z^{-a} \Lambda) \cap (z^{-a'} \Lambda')} \,\,= \,\,
\min\left( a+N_\Lambda, a' + N_{\Lambda'} \right).
\]
By Lemma \ref{lem:3equiv}, for any integers $a_1, a_2, \dots, a_s$,
the image under the map $\Psi_M$ of the retraction $\,
r_M(z^{-a_1}\Lambda_1 \cap \cdots \cap z^{-a_s}\Lambda_s) \,$
 coincides with the tropical linear combination
\[
( a_1 \odot \Psi_{M}(r_M(\Lambda_1)) ) \oplus \dots \oplus (a_s \odot \Psi_{M}(r_M(\Lambda_s)) ).
\]
The simplicial complex structure of $[M]$ coincides with the standard triangulation of the tropical linear space $L_p$,
which induces the simplicial complex structure on the lattice points in the tropical polytope.  Hence the retraction of
the min-convex hull onto the membrane coincides with the standard triangulation of the tropical polytope.
\end{proof}

\begin{algorithm}[htb]
  \dontprintsemicolon

  \KwIn{matrices $M_1,\ldots,M_s\in\GL_d(K)$ and a $d\times n$ matrix $M$ over $K$ with rank $d$}
  \KwOut{retraction $r_M(\minconv(\Lambda_1,\dots,\Lambda_s))$ onto the membrane $[M]$, \break
   where $\Lambda_i = \image_R(M_i)$ for $i = 1,\ldots,s$.}
  \For{$i\leftarrow 1,2,\dots,s$}{
    $\Psi_M(r_M(\Lambda_i))\leftarrow$ tropical sum of the rows of $\val(M_i^{-1}\cdot M)$\;
  }
  \Return $\,\,\tconv(\Psi_M(r_M(\Lambda_1)),\Psi_M(r_M(\Lambda_2)),\dots,\Psi_M(r_M(\Lambda_s)))$\;
  
\caption{Retraction of a min-convex hull in $\BT_d$ onto a given membrane. \label{algo:retract}}
\end{algorithm}

\begin{example} \label{ex:ProbA} (Illustration of  Algorithm \ref{algo:retract}) \ \ 
  We consider the three lattices $\Lambda_1, \Lambda_2,
  \Lambda_3$   in the Bruhat--Tits building $\BT_3$ 
  which are represented by the invertible 
  $3 \times 3$-matrices
  \[ M_1 \, = \, 
  \begin{pmatrix}
    1 &  z^5 &  z^{-3} \\
    z^4 &  z &  z^{-3} \\
    z^{-3} & z^2 &  z^{-3}
  \end{pmatrix}   ,  \;
  M_2 \, = \, 
  \begin{pmatrix}
    z^2 &  z^{-2} &  z^2 \\
    z^3  &  z^5  &  z^5 \\
    1  &  1  &  z^4
  \end{pmatrix}  , \;
  M_3 \, = \, 
  \begin{pmatrix}
    z^2&  z^{-1} &  z \\
    z^{-2} &  z^{-3} &  z^3 \\
    z^3 &  z &  1
  \end{pmatrix}  .
  \]
  Set $M:=(M_1,M_2,M_3)$. Then the vectors $\Psi_M(r_M(\Lambda_1))$, $\Psi_M(r_M(\Lambda_2))$, and
  $\Psi_M(r_M(\Lambda_3))$ are the precisely the rows of the $3 \times 9$-matrix $V$ in~\eqref{39matrix}. That matrix
  was analyzed in Examples \ref{ThreeNine} and \ref{ThreeNineAgain}.  Hence the tropical convex hull (of the rows) of
  $V$ is the tropical polygon $P$ in Figure~\ref{fig:polygon}.
  
  The $31$ lattices in $P$ are encoded by the $31$ lattice points in Figure \ref{fig:polygon}, or by the $31$ lattice
  vectors listed in Example \ref{ThreeNineAgain}.  If $u = (u_1,u_2,\dots,u_9) \in \Z^9$ is one these vectors then the
  corresponding lattice $\Lambda \subset K^3$ is generated by the nine columns of the $3 \times 9$-matrix
  \[
  M \cdot \diag(z^{-u_1},z^{-u_2},\ldots,z^{-u_9}) \, .
  \]
  The underlined coordinates of $u$ give the lexicographically first basis $\{i,j,k\}$ of the 
  matroid $M_u$. This writes $\Lambda$ as the column lattice of the matrix $ M \cdot
  \diag(z^{-u_i},z^{-u_j},z^{-u_k})$. \qed
\end{example}

\subsection{Computing min-convex hulls in $\BT_d$}

Algorithm \ref{algo:retract} would compute the min-convex hull in $\BT_d$ if we input a membrane that contains it.
Algorithm \ref{algo:A:minconvex} below iteratively finds such a membrane, starting from the membrane $[M]$ spanned by
the given generators of $\Lambda_1, \dots, \Lambda_s$.  The idea is to compute the retraction $P$ of the min-convex hull
onto $[M]$, to identify the fiber over every lattice in $P$, and then to enlarge our membrane by the fibers.

As seen in the proof of Proposition \ref{prop:Algo1}, each lattice in
the desired convex hull,
\[
z^{-a_1}\Lambda_1 \cap \cdots \cap z^{-a_s} \Lambda_s 
\,\,\in\,\, \minconv(\Lambda_1, \ldots, \Lambda_s) ,
\]
is mapped by the composition $\Psi_M \circ r_M$ to the tropical linear combination
\[
a_1 \odot \Psi_M(r_M(\Lambda_1)) \,\oplus\, \cdots 
\,\oplus\,  a_s \odot \Psi_M(r_M(\Lambda_s)) 
\,\,\in\,\, P.
\]
Our aim is to list all lattices in the fiber $\smallSetOf{ \Lambda \in \minconv(\Lambda_1, \ldots,
  \Lambda_s}{\Psi_M(r_M(\Lambda)) = v }$ over an lattice point $v \in P$.  There are infinitely many ways to write $v$
as an integer tropical linear combination of $\Psi_M(r_M(\Lambda_1)), \dots, \Psi_M(r_M(\Lambda_s))$.  However, since
the min-convex hulls in $B_d$ are finite, the fibers under the retraction are finite, too.  We can make sure that the
loop in step \ref{step:fiber} is finite, as follows.  For a fixed $v \in P$, let $C_v$ be the set of coefficients $a \in
\Z^s$ such that $v = \bigoplus_{i=1}^s \left( a_i \odot \Psi_M(r_M(\Lambda_i)) \right)$.  Then $C_v$ is a partially
ordered set with $a \leq b$ in $C_v$ if $a_i \leq b_i$ for all $i = 1,\ldots,s$.  This partial order is compatible with
the inclusion order on the fiber, i.e.~$a \leq b$ implies $\bigcap_{i=1}^s \left( z^{-a_i} \Lambda_i \right) \subseteq
\bigcap_{i=1}^s \left( z^{-b_i} \Lambda_i \right)$.  Note that if $a,b \in C_v$ then $a \oplus b \in C_v$, so there is a
unique minimal element in $C_v$.  Starting from the unique minimal element in $C_v$, we do a finite depth-first-search
on the Hasse diagram of $C_v$ to enumerate the fiber over $v$.  At every step, we increment a coordinate by $1$ if the
new lattice is strictly larger.  Otherwise, further incrementing that coordinate will not give us new lattices in the
fiber, so we abandon that branch and backtrack.  In this manner we reach all elements in the fiber without going through
an infinite loop.  As a byproduct, Algorithm \ref{algo:A:minconvex} produces a membrane $[M']$ which contains the
min-convex hull.

 \begin{algorithm}[htb]
  \dontprintsemicolon

  \KwIn{matrices $M_1,M_2,\ldots,M_s\in\GL_d(K)$}
  \KwOut{$\minconv(\Lambda_1,\ldots,\Lambda_s)$ in $\BT_d$, where $\Lambda_i = \image_R(M_i)$}

  \smallskip

  $M\leftarrow(M_1, \dots, M_s)\in K^{d\times ds}$\;
  $M' \leftarrow M$\;
  $P\leftarrow r_M(\minconv(\Lambda_1,\ldots,\Lambda_s))$, computed by Algorithm~\ref{algo:retract}

\lnl{step:integer}\ForEach{lattice point $v \in P$}{
                           $\Lambda \leftarrow  R\{z^{-v_j} f_j\}$ where $f_j$ is the $j^{th}$ column of $M$\;
\lnl{step:fiber}      \ForEach{$a \in \Z^s$ such that $v = \bigoplus_{i=1}^s \left( a_i \odot \Psi_M(r_M(\Lambda_i)) \right)$}{
        \If{ $\Lambda \subsetneq \bigcap_{i=1}^s \left( z^{-a_i} \Lambda_i \right)$  }{
          Augment the columns of $M'$ with minimal generators \break of $\bigcap_{i=1}^s \left( z^{-a_i} \Lambda_i \right)$
          that are not in $\Lambda$.\;
                }
        }
     }
     $P'\leftarrow r_{M'}(\minconv(\Lambda_1,\ldots,\Lambda_s))$, computed by Algorithm~\ref{algo:retract}

  \Return $P'$\;
  \caption{Min-convex hull in the Bruhat--Tits building $\BT_d$.\label{algo:A:minconvex}}
\end{algorithm}

\begin{figure}\centering
  \includegraphics[width=1 \linewidth]{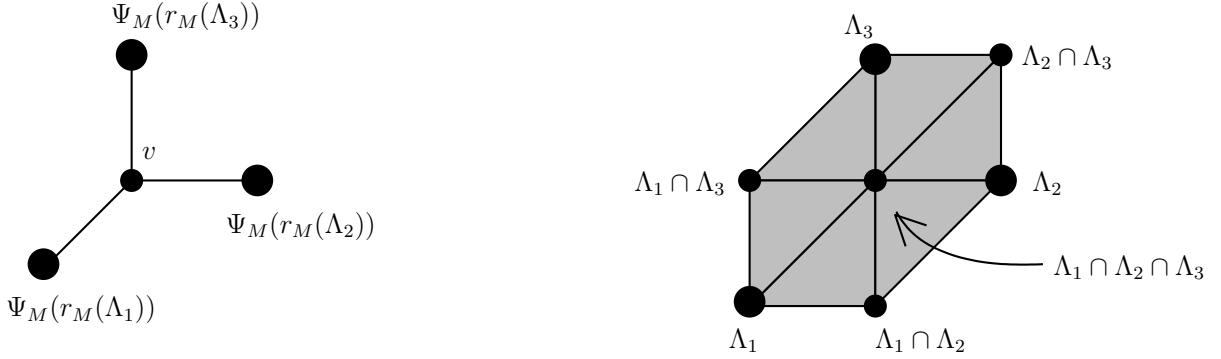}
  \caption{The two iterations of Algorithm \ref{algo:A:minconvex} 
  for $\Lambda_1,\Lambda_2,\Lambda_3$ as in Example \ref{ex:3lattices}.  }
 \label{fig:3lattices}
\end{figure}

\begin{example}
\label{ex:3lattices}
We illustrate Algorithm \ref{algo:A:minconvex}
by computing the min-convex hull of three points in the Bruhat--Tits building $\BT_3$.
The input points are given by the three invertible matrices
\[ M_1 \, = \, 
\begin{pmatrix}
  z & \! 0 & \! 0\\
  0 & \! 1 & \! 0 \\
  0 &\! 0 & \! 1
\end{pmatrix} \! \! ,  \;
M_2 \, = \, 
\begin{pmatrix}
  \! z & \! 1 & \! 1 \\
  \! 0  & \! 1  & \! 0 \\
  \! 0  & \! 0  & \! 1
\end{pmatrix} \!\! , \;
M_3 \, = \, 
\begin{pmatrix}
  \! z& \! 1 & \! 4 \\
  \! 0 & \! 2 & \! 5 \\
  \! 0 & \! 3 & \! 6
\end{pmatrix} \!\! .
\]
We start with the membrane spanned by $M = (M_1, M_2, M_3)$, and hence with
\[
\left(
\begin{array}{c}
\Psi_M (r_M(\Lambda_1)) \\
\Psi_M (r_M(\Lambda_2)) \\
\Psi_M (r_M(\Lambda_3)) \\
\end{array}
 \right) \,\,\, = \,\,\,
 \left(
 \begin{array}{ccccccccc}
 0&0&0&0&-1&-1&0&-1&-1\\
 0&-1&-1&0&0&0&0&-1&-1\\
 0&-1&-1&0&-1&-1&0&0&0
 \end{array}
 \right)
\]
The tropical convex hull of these three row vectors has precisely one more lattice point:
\begin{equation*}
\begin{array}{rcl}
v  &= &  (0,-1,-1,0,-1,-1,0,-1,-1)\\ & = &       \Psi_M (r_M(\Lambda_1))      \,\oplus \, \Psi_M ( r_M(\Lambda_2)) \\
& = &  \Psi_M (r_M(\Lambda_1))      \,\oplus \, \Psi_M ( r_M(\Lambda_3)) \\
& = & \Psi_M (r_M(\Lambda_2))      \,\oplus \, \Psi_M ( r_M(\Lambda_3)) \\
&= &  \Psi_M (r_M(\Lambda_1))  \,\oplus \, \Psi_M ( r_M(\Lambda_2)) \,\oplus \, \Psi_M ( r_M(\Lambda_3)) .      \end{array}
\end{equation*}
The set $C_v$ consists of the vectors
$(0,0,a), (0,b,0)$ and $(c,0,0)$ where $a,b,c \in \N$.
The unique  minimal element is $(0,0,0)$.
As its corresponding lattice $\,z R^3 = \Lambda_1 \cap \Lambda_2 \cap \Lambda_3\,$
lies in $[M]$, this point adds no new columns to $M'$.
Since $\Lambda_1 \cap \Lambda_2 \cap z^{-1} \Lambda_3 = \Lambda_1 \cap \Lambda_2 \cap z^{-2} \Lambda_3$, all  lattices $\Lambda_1 \cap \Lambda_2 \cap z^{-a} \Lambda_3$ are 
identical for $a \geq 1$. So we can abandon
 the branch $(0,0,a)$ in $C_v$ after $(0,0,1)$.  Similarly, we only need to consider up to $(0,1,0)$ and $(1,0,0)$.

After comparing $z R^3$ with the lattices $\Lambda_1 \cap \Lambda_2 \cap z^{-1}\Lambda_3$, $\Lambda_1 \cap z^{-1} \Lambda_2 \cap \Lambda_3$  and $ z^{-1}\Lambda_1 \cap \Lambda_2 \cap \Lambda_3$ respectively, we augment the columns of $M'$ with the three vectors:
\begin{equation*}
\begin{array}{rcl}
(0,1,-1) & \in &  (\Lambda_1 \cap \Lambda_2 \cap z^{-1}\Lambda_3) \setminus z R^3\\
(0,1,2) & \in &  (\Lambda_1 \cap z^{-1} \Lambda_2 \cap \Lambda_3) \setminus z R^3 \\
(3,2,1) & \in &   (z^{-1}\Lambda_1 \cap \Lambda_2 \cap \Lambda_3) \setminus z R^3 .\\
\end{array}
\end{equation*}

With this new matrix $M'$, the images of $\Lambda_i$ under the map $\Psi_{M'}$ become:
\[
\left(
\begin{array}{c}
\Psi_{M'}(\Lambda_1) \\
\Psi_{M'}(\Lambda_2) \\
\Psi_{M'}(\Lambda_3) \\
\end{array}
 \right) \,\, = \,\, 
 \left(
 \begin{array}{cccccccccccc}
 0&0&0&0&-1&-1&0&-1&-1&0&0&-1\\
 0&-1&-1&0&0&0&0&-1&-1&0&-1&0\\
 0&-1&-1&0&-1&-1&0&0&0&-1&0&0
 \end{array}
 \right)
\]
This new membrane $[M']$ contains all the lattices in the min-convex hull of $\Lambda_1$, $\Lambda_2$, and $\Lambda_3$.  The tropical convex hull of the three rows contains four other distinct lattice points:
\[
\begin{array}{rcl}
\Psi_{M'}(\Lambda_1 \cap \Lambda_2) & = & \Psi_{M'}(\Lambda_1) \oplus \Psi_{M'}(\Lambda_2), \\
\Psi_{M'}(\Lambda_1 \cap \Lambda_3) & = & \Psi_{M'}(\Lambda_1) \oplus \Psi_{M'}(\Lambda_3), \\
\Psi_{M'}(\Lambda_2 \cap \Lambda_3) & = & \Psi_{M'}(\Lambda_2) \oplus \Psi_{M'}(\Lambda_3), \\
\Psi_{M'}(\Lambda_1 \cap \Lambda_2 \cap \Lambda_3) & = & \Psi_{M'}(\Lambda_1) \oplus \Psi_{M'}(\Lambda_2) \oplus \Psi_{M'}(\Lambda_3).
\end{array}
\]
The simplicial complex $\,\minconv(\Lambda_1,\Lambda_2,\Lambda_3)\,$ is shown
on the right in Figure \ref{fig:3lattices}.
\qed
\end{example}

Algorithm \ref{algo:A:minconvex} solves Computational Problem A in the min-convex case.  Computing max-convex hulls
reduces to computing min-convex hulls, as shown in Algorithm \ref{algo:A:dualconvex}.

\begin{algorithm}[htb]
  \dontprintsemicolon

  \KwIn{matrices $M_1,M_2,\dots,M_s\in\GL_d(K)$}
  \KwOut{$\maxconv(\Lambda_1,\dots,\Lambda_s)$ in $\BT_d$, where $\Lambda_i = \image_R(M_i)$}

  \smallskip
  Run Algorithm \ref{algo:A:minconvex} with input matrices $M_1^{-T}, \dots, M_s^{-T}$.\;
  \Return $\minconv(\Lambda_1^*, \dots, \Lambda_s^*)$.
  \caption{Max-convex hull in the Bruhat--Tits building $\BT_d$.
    \label{algo:A:dualconvex}}
\end{algorithm}

The correctness of Algorithm \ref{algo:A:dualconvex} follows from Lemma \ref{lem:dual}, which implies that the
simplicial complex structure of the max-convex hull of $\Lambda_1, \ldots, \Lambda_s$ is identical to the simplicial
complex structure of the min-convex hull of $\Lambda_1^*, \ldots, \Lambda_s^*$.  Our procedure exhibits a matrix of
basis vectors for each lattice in $\,\minconv(\Lambda_1^*, \ldots, \Lambda_s^*)$.  We take the inverse transpose of that
matrix to get a basis matrix for the corresponding lattice in $\,\maxconv(\Lambda_1, \ldots, \Lambda_s)$.

\subsection{Implementations}

We now come to question of how our convex hull algorithms can be used in practice, and what implementations are within
reach.  We largely focus on the operator ``$\tconv$'' which is crucial in Algorithm~\ref{algo:retract}, which in turn is
called twice in Algorithm~\ref{algo:A:minconvex}.  Its output form (and hence also the form of the final output of the
algorithm) were left deliberately vague, as there are several choices for how ``$\tconv$'' can be realized.  Firstly,
there is a direct polyhedral approach for computing tropical convex hulls which is based on the following result from
\cite[Section~4]{DS}: The tropical convex hull of $n$ points in $\TP^{s-1}$ arises as the polyhedral complex of bounded
faces in an ordinary convex polyhedron defined by $ns$ linear inequalities in $\R^{n+s}$.  This method is implemented in
\texttt{polymake}~\cite{GJ}.  The details of this implementation together with extensive tests are the topic
of~\cite{HJP}.  Secondly, one can use the algebraic algorithm based on resolutions of monomial ideals which was
described in~\cite{BY}.  A \texttt{Macaulay2}/\texttt{Maple} implementation is available from the third author.  In the
planar case, $s=3$, specific techniques from computational geometry can be used to design alternative, faster
algorithms; see~\cite{J}.

In view of tropical polytope duality \cite[Theorem 23]{DS}, we can choose if we want to compute the tropical convex hull
of $n$ points in $\TP^{s-1}$ or of $s$ points in $\TP^{n-1}$.  If $s\le 3$ then, due to the specialized algorithms
mentioned above, it is easier to compute the tropical convex hull of $n$ points in $\TP^{s-1}$.  The output of both, the
polyhedral and the algebraic algorithms, returns a tropical polytope $P$ decomposed into cells as in (\ref{cell}).

Enumerating the lattices in Step \ref{step:integer} then requires to list all the lattice points in the ordinary
polytopes corresponding to the types.  In higher dimensions this can be an arduous task, due to the sheer size of the
output.  Hence, depending on the application intended, it may be advisable to stick with the output of the previous
stage as a compressed description of the set of lattices.  From each type we can read off the matroid $M_u$ which
specifies the set of apartments (spanned by the columns of $M$) containing that type.  In Example \ref{4leaftree}, these
matroids $M_u$ are the sets of pairs such as $\{af, bf, cf, df, ef, fg, fh\}$.  

\begin{table}[ht]\centering
  \caption{Timings in seconds for computations with ``$\tconv$''  in \texttt{polymake}.
  The parameters $d$ and $s$ indicate the
    size of the problem, that is, computing the min-convex hull of $s$ lattices represented by $d\times
    d$-matrices. $N$ is the number of samples tested, and the last four columns contain 
     basic statistics. \label{tab:timings}}
  \renewcommand{\arraystretch}{0.9}
  \begin{tabular*}{.9\linewidth}{@{\extracolsep{\fill}}rrrrrrr@{}}\toprule
    $d$ & $s$ & $N$ & mean & stddev & $\min$ & $\max$\\ \midrule
     3 &  2 & 50 & 0.18 & 0.02 & 0.15 & 0.21\\
     3 &  3 & 50 & 0.55 & 0.14 & 0.31 & 0.88\\
     3 &  4 & 50 & 2.02 & 0.94 & 0.68 & 5.47\\
     3 &  5 & 50 & 7.73 & 2.77 & 2.92 & 14.25\\
     3 &  6 & 50 & 18.27 & 8.21 & 5.40 & 45.78\\
     3 &  7 & 50 & 38.78 & 15.21 & 9.30 & 77.65\\
     3 &  8 & 50 & 69.39 & 23.21 & 30.02 & 124.05\\
     3 &  9 & 50 & 119.63 & 41.90 & 27.66 & 243.25\\
     3 & 10 & 50 & 231.17 & 111.22 & 71.89 & 594.95\\
     4 &  2 & 50 & 2.75 & 1.30 & 0.79 & 6.07\\
     4 &  3 & 50 & 62.79 & 42.54 & 12.20 & 178.97\\
     4 &  4 & 50 & 827.37 & 624.19 & 93.74 & 3017.19\\
     4 &  5 & 18 & 5994.15 & 4986.38 & 648.14 & 21018.16\\
     4 &  6 & 5 & 35823.43 & 21936.56 & 4846.15 & 67876.56\\
     4 &  7 & 5 & 28266.78 & 15773.94 & 9193.69 & 55891.92\\ \bottomrule
  \end{tabular*}
\end{table}

To give a sense of the running time of tropical convex hull code, in
Table~\ref{tab:timings} we list a few timings of \texttt{polymake} computations.  The samples were generated at random
from $s\times sd$-matrices with integer entries ranging from $0$ to $9$.  The algorithm uses the general polyhedral approach without the enumeration of lattice points.  The individual
timings vary quite a bit, and individual examples with smaller parameters may need more time than
other examples with larger parameters.
 Nonetheless, the reader should get an idea. For more comprehensive tests 
 we refer to \cite{HJP}.  Hardware: AMD
4200+X2, 4423 bogomips, 2GB main memory.  Software implemented in \texttt{polymake} 2.3 on SuSE Linux 10.0.

\section{Further Algorithms and Perspectives}\label{sec:perspectives}

\noindent
We now consider Computational Problem B: Determine the intersection of $s$ membranes.
The input consists of matrices $M_1,\ldots, M_s$, each having $d$ linearly independent rows over $K = \C(\!(z)\!)$.  Here $M_i$ represents the
membrane $[M_i]=[(f_{i1}, \dots, f_{id})]$, where $f_{ij}$ is the $j$th column of the matrix $M_i$.  The 
intersection $\,[M_1]\cap[M_2]\cap\dots\cap[M_s]\,$ is a locally finite simplicial complex of dimension $\leq d-1$. It
may be finite or infinite, depending on the input. We will compute this intersection as a tropical polytope over $( \R \cup \{\infty\} , \oplus, \odot)$.

Obviously,  $[M_1]\cap[M_2]\cap\dots\cap[M_s]$ is contained in the union
$[M_1]\cup[M_2]\cup\dots\cup[M_s]$, which in turn is contained in the membrane $\,[(M_1,M_2,\dots,M_s)]$.  By
Theorem~\ref{thm:membrane}, this membrane is isomorphic, as a simplicial complex, to the standard
triangulation of the tropicalization $L_p(M)$ of the row space of $M=(M_1,M_2,\dots,M_s)$.  In view of Theorem
\ref{thm:tropical_linear_space}, we may regard $L_p(M)$ as a polytope in the compactified tropical projective space
$\,\overline{\TP}^{sd-1}$.

Our computations take place inside this tropical linear space $L_p(M)$, which we represent as the tropical convex hull of
the cocircuits $\,p(\sigma *)\,$ that are derived from the matrix~$M$.  The $k$-th column vector $f_{ik}$ of the $i$-th
input matrix $M_i$ corresponds to the cocircuit $\,p(\sigma *)\,$ where $\sigma$ is the $(d{-}1)$-subset of
$\{1,2,\ldots,sd\}$ which indexes all columns of $M_i$ other than $f_{ik}$ inside $M$. 
This special cocircuit is abbreviated by
$\, C_{ik} \, := \, \val \bigl(\,\text{the $k$-th row $M_i^{-1} \cdot M$} \, \bigr) $.
  Consider the subpolytope of $L_p(M)$ spanned by
  the $d$ special cocircuits arising from $M_i$:
\[
L_p^M(M_i) \quad = \quad \tconv\{C_{i1}, \ldots, C_{id}\}.
\]
This tropical polytope with its standard triangulation is isomorphic to the membrane $[M_i]$.
Intersecting these subpolytopes $L_p(M_i)$ inside $L_p(M)$ 
solves Computational Problem $B$.
 
The intersections of arbitrary tropical polytopes are tropical polytopes again~\cite[Proposition 20]{DS}.  Here,
however, the situation is even easier since the subpolytope $L_p^M(M_i)$, as an ordinary polytopal complex, is a
subcomplex of $L_p(M)$.  We summarize our findings in Algorithm~\ref{algo:C:minconvex}.  Our remarks concerning the
output of Algorithm~\ref{algo:A:minconvex} apply accordingly.

\begin{algorithm}
  \dontprintsemicolon

  \KwIn{Matrices $M_1,M_2,\dots,M_s\in\GL_d(K)$}
  \KwOut{Intersection $[M_1]\cap[M_2]\cap\dots\cap[M_s]$ of membranes in $\BT_d$}

  \smallskip
  
  $M \leftarrow(M_1,M_2,\dots,M_s)$\;
  $C \leftarrow sd\times sd\text{-matrix of cocircuits of } M$\;
  $L_p(M) \leftarrow \tconv\{c_{11},\dots,c_{ss}\}$\;
  \For{$k\leftarrow 1,2,\dots,s$}{
    $L_p^M(M_i) \leftarrow \tconv\{c_{i1}, \dots, c_{id}\}$\;
  }
  $I\leftarrow\emptyset$\;
  \ForEach{cell $C$ in $L_p(M)$}{
    \If{$C\subseteq L_p^M(M_i)$ for all $i$}{
      $I\leftarrow I\cup C$\;
    }
  }
  \Return $ I $\;

  \caption{Intersection of membranes in the affine building $\BT_d$
\label{algo:C:minconvex}}
\end{algorithm}

We now examine the special case of Computational Problem B where each input matrix $M_i$ is square. Here our problem is
to compute the intersection of $s$ apartments in $\BT_d$.  Since apartments are both min- and max-convex, the
intersection of apartments is also min- and max-convex.  This establishes the connection between Computational Problem B
and the classical notion of convexity in Remark \ref{rem:convex-def}. The set of all chambers which are fully contained
in the intersection of apartments is convex in the sense of Remark \ref{rem:convex-def}.  Note that (the vertex set of)
every convex set of chambers within some apartment of $\BT_d$ arises in this manner, namely as the output of Algorithm
\ref{algo:C:minconvex} for some square matrices $M_1, \ldots,M_s$.  Identifying one of the apartments with $\TP^{d-1}$,
we see that the result of this computation is a subset of $\TP^{d-1}$ which is both min-convex and max-convex. This
implies that the intersection of apartments is an ordinary convex polytope of the special form (\ref{cell}).

Recent work of Alessandrini \cite{Ale} suggests the following alternative method this computation, which more efficient
than applying Algorithm \ref{algo:C:minconvex} to square matrices.  Our point of departure towards Alessandrini's method
is the following question: \emph{Given $M \in \GL_d(K)$, how can we decide whether the standard lattice $R^d$ lies in
  the apartment $[M]$, i.e. whether $R^d$ has an $R$-basis of the form $\{z^{a_1} f_1, z^{a_2}f_2,...., z^{a_d}f_d \}$
  for some integers $a_1,a_2,\ldots,a_d $?}

To answer this question, we compute the tropical $d \times d$-matrix
\begin{equation}
  \label{defE(M)}
  E(M) \quad := \quad   \val(M) \odot \val(M^{-1}) \, .
\end{equation}
Here $\odot $ means that the matrix product is evaluated in the min-plus algebra.  Note that each diagonal entry of
$E(M)$ is non-negative.  The following lemma is easy to derive:

\begin{lemma} \label{lem:liesIn}
 The following are equivalent for a matrix $M \in \GL_d(K)$:
 \begin{enumerate}
 \item The standard lattice $R^d$ lies in the apartment $[M]$.
 \item By scaling the columns of $M$ with powers of $z$, we can get a matrix $G$ in $R^{d \times d}$ whose constant term
   $G(0) \in \C^{d \times d}$ is invertible.
 \item Each entry $e_{ij}(M)$ of the matrix E(M) is non-negative.
 \end{enumerate}
\end{lemma}

We now change the question as follows. Let $u_1,\ldots,u_d$ be unknown integers.  Under what condition on these integers
is the scaled standard lattice $R \{ z^{u_1} e_1, ..., z^{u_d} e_d \}$ in the apartment $[M]$? This question is
equivalent to asking whether the standard lattice $R^d$ lies in the apartment $\,[\,\diag(z^{-u}) \cdot M\,]$, where
$\diag(z^{-u}) = \diag(z^{-u_1}, \ldots, z^{-u_d})$.  By applying Lemma \ref{lem:liesIn} to the matrix $\,\diag(z^{-u})
\cdot M\,$ in place of $M$, we obtain the following result.

\begin{cor} \label{cor:Ales}
  The lattice $R\{ z^{u_1} e_1, ..., z^{u_d} e_d \}$ lies in the apartment $[M]$ if and only if
  \begin{equation}
    \label{InversionDomain}
    u_j - u_i \,\,\leq \,\, e_{ij}(M) \qquad \hbox{for}\,\,\,
    i,j = 1,2,\ldots,d. 
  \end{equation}
\end{cor}

The linear inequalities \eqref{InversionDomain} in the unknowns $u_1,\ldots,u_d$ defines a convex subset of $\TP^{d-1} $ which is both an ordinary polytope and a tropical polytope.  Corollary \ref{cor:Ales} is
essentially equivalent to Theorem 4.7 in \cite{Ale}. Alessandrini refers to the polytope (\ref{InversionDomain}) as the
\emph{inversion domain} associated with the tropical matrix product in (\ref{defE(M)}); see \cite[Proposition 3.4]{Ale}.
We conclude that the intersection of the two apartments $[M]$ and $[\diag(1,\ldots,1)]$ equals the standard
triangulation of the inversion domain, which is specified by the inequalities (\ref{InversionDomain}).

We now present our second method, to be called \emph{Alessandrini's Algorithm},
for Computational Problem B in the special case of apartments.
The input consists of $s$ invertible matrices $M_1,M_2,\dots,M_s$ over $K$,
and the output is the intersection $[M_1]\cap\dots\cap[M_s]$ of apartments.
After multiplying each matrix on
the left by $M_1^{-1}$, we may assume that $M_1$ is the identity matrix.  Then the desired intersection is the standard
triangulation of the polytope specified by the inequalities (\ref{InversionDomain}) where $M$ runs over
$\{M_2,\dots,M_k\}$.  Alessandrini's Algorithm is summarized by the following refinement of
Proposition~\ref{prop:convex}.

\begin{theorem}\label{thm:alessandrini}
  The intersection of apartments $[M_1] \cap \cdots \cap [M_s]$ in the Bruhat--Tits building $\BT_d$ is the standard
  triangulation of a polytope of the form (\ref{cell}), namely, the polytope
  \[
  \SetOf{ u \in \TP^{d-1} }{ u_j - u_i \leq e_{ij}(M_k) \text{ for } i,j = 1,\dots,d \text{ and } k = 2,\dots,s } \, .
  \]
\end{theorem}

\section*{Conclusion}

\noindent
We have demonstrated that tropical convexity is a useful tool for computations with affine buildings.  Given the
ubiquitous appearance of affine buildings in mathematics, we are optimistic that our approach can be of interest for a
wide range of applications.  Such applications may arise in fields as diverse as geometric topology \cite{Ale}, number
theory \cite{Fal,Set}, algebraic geometry \cite{Ka,KT}, representation theory \cite{Gor}, harmonic analysis \cite{Par},
and differential equations \cite{Cor}.  Experts in combinatorial representation theory may find it interesting to
generalize our constructions and algorithms to affine buildings of other types.  This will require to investigate, for
instance, the $B_n$-analogs of tropical polytopes.


\def\cprime{$'$}
\providecommand{\bysame}{\leavevmode\hbox to3em{\hrulefill}\thinspace}
\providecommand{\MR}{\relax\ifhmode\unskip\space\fi MR }
\providecommand{\MRhref}[2]{%
  \href{http://www.ams.org/mathscinet-getitem?mr=#1}{#2}
}
\providecommand{\href}[2]{#2}

\end{document}